\documentclass[11pt]{amsart}

\usepackage{amsmath}
\usepackage{amsfonts}
\usepackage{amsthm}
\usepackage{euscript}
\usepackage{amssymb}
\usepackage{graphicx}
\usepackage{mathrsfs}
\usepackage{enumitem}
\usepackage{indentfirst}
\usepackage{csquotes}
\usepackage{bbm}

\numberwithin{equation}{section}

\newcommand{\field}[1]{\mathbb{#1}}
\newcommand{\N}{\field{N}}
\newcommand{\C}{\mathbb{C}}

\newcommand{\Rn}{\mathbb{R}^{n-1}}
\newcommand{\Zn}{\mathbb{Z}^{n-1}}
\newcommand{\Tn}{\mathbb{T}^{n-1}}

\newcommand{\la}{\lambda}

\newcommand\norm[1]{\left\lVert#1\right\rVert}

\newcommand{\lpnorm}[3]{\norm{#1}_{L^{#2}(#3)}}
\newcommand{\altlpnorm}[2]{\norm{#1}_{#2}}
\newcommand{\mixlpnorm}[4]{\norm{#1}_{L^{#2}_{t} L^{#3}_{x}(#4)}}

\newcommand{\litlpnorm}[3]{\norm{#1}_{l^{#2}(#3)}}

\newcommand{\Lpopnormalt}[5]{\norm{#1}_{L^{#2}_{#3} \to L^{#4}_{#5} }}

\newcommand{\lpn}[2]{\norm{#1}_{#2}}
\newcommand{\supp}[1]{\text{supp}\, #1}
\newcommand{\fsupp}[1]{\text{supp}\, \widehat{#1}}
\newcommand{\R}[1][]{\mathbb{R}^#1}
\newcommand{\Z}[1][]{\mathbb{Z}^#1}
\newcommand{\T}[1][]{\mathbb{T}^#1}
\newcommand{\parab}[1][]{\mathbb{P}^#1}

\newcommand{\laplac}[1]{\Delta_{#1}}

\newcommand{\Sl}{S_{\la}}

\newcommand{\extensionop}[2]{E_{#1}^{\la}{#2}}
\newcommand{\dK}[1]{\delta_{K}^{#1}}
\newcommand{\Sq}[3]{Q_{#3}^{#2}#1}

\newtheorem{proposition}{Proposition}[section]
\newtheorem{theorem}{Theorem}[section]

\newtheorem{corollary}{Corollary}[theorem]
\newtheorem{lemma}[theorem]{Lemma}
\newtheorem{remark}{Remark}[section]

\begin{document}

\title{Lossless Strichartz Estimates on Rectangular Tori over Short Time Intervals}
\author{Connor Quinn}

\begin{abstract}
    We prove lossless Strichartz estimates at the critical exponent $q_c = \frac{2(n+1)}{n-1}$ and the endpoint exponent pair $\left(2,\frac{2(n-1)}{n-3}\right)$ for the Schr\"{o}dinger equation on rectangular tori of dimension $n-1$ with frequency localized initial data on small time windows with length depending on the frequency parameter $\lambda \gg 1$. This work improves on previous results of Burq-Gerard-Tzvetkov \cite{bgtstrichartzmanifold} for general compact manifolds and Blair-Huang-Sogge \cite{BlairHuangSoggeStrichartz} and Huang-Sogge \cite{HuangSoggeStrichartz} for compact manifolds of nonpositive curvature by showing that we can take the time interval to be of length $\lambda^{-\delta_n}$ for some $\delta_n \in (0,1)$ depending on the dimension of the torus. The proof uses Tao’s bilinear restriction theorem for the paraboloid from \cite{taobilinear} and an iteration argument over multiple scales that takes advantage of improved kernel estimates arising from the Fourier multiplier decomposition we employ.
\end{abstract}
\maketitle

\section{Introduction}
Let $(M,g)$ be a compact Riemannian manifold without boundary of dimension $d \geq 2$. We will let $n$ denote the space-time dimension (i.e the dimension of $M \times \mathbb{R}$) so that $d = n-1$. Let $\laplac{g}$ denote the Laplace-Beltrami operator associated to the metric and let
\begin{equation}
    u(x,t) = e^{-it\laplac{g}}f(x)
\end{equation}
be the solution to the Schr\"{o}dinger equation on $M \times \R{}$
\begin{equation} \label{schrodinger equation}
    i\partial_{t} u(x,t) = \laplac{g}u(x,t), \ u(x,0) = f(x).
\end{equation}
The operator $e^{-it\laplac{g}}$ is called the Schr\"{o}dinger propagator. 

One set of tools used to analyze the space-time dispersive properties of these solutions are Strichartz estimates, which on a compact manifold take the form

\begin{equation} \label{abstract strichartz}
    \mixlpnorm{u}{p}{q}{M \times [0,1]} \lesssim \norm{f}_{H^s(M)}
\end{equation}
where the mixed norm is defined as
\begin{align*}
    \mixlpnorm{F}{p}{q}{M \times I} = \Big(\int_{I}\norm{F(\cdot, t)}_{L^{q}(M)}^pdt\Big)^{\frac{1}{p}}
\end{align*}
for some interval $I \subset \R{}$
and 
\begin{align*}
    \norm{f}_{H^s(M)} = \lpnorm{\big( I + P\big)^{s}f}{2}{M}, \ \ \text{with} \ P = \sqrt{-\laplac{g}}, 
\end{align*}
is the standard $L^2$-based Sobolev space on $M$. The case where the exponent pair $(p,q)$ in (\ref{abstract strichartz}) is Keel-Tao admissible, meaning it satisfies the relation

\begin{equation}
    d\left(\frac{1}{2} - \frac{1}{q}\right) = \frac{2}{p}, \ \text{and} \ 2 < q \leq \frac{2d}{d-2} \  \text{if} \ d \geq 3, \ \text{or} \ 2 < q < \infty \ \text{if} \ d = 2,
\end{equation}
is of particular interest. 

The study of such inequalities in this setting was first initiated by Burq-Gerard-Tzvetkov in \cite{bgtstrichartzmanifold} where they showed that (\ref{abstract strichartz}) holds for admissible $(p,q)$ if $s \geq \frac{1}{p}$. To show this, they actually proved a stronger uniform inequality on a short time interval, namely that
\begin{equation} \label{short time dyadic estimate}
    \mixlpnorm{e^{-it\laplac{g}} \beta(P/ \la)f}{p}{q}{M \times [0, \la^{-1}]} \leq C \altlpnorm{f}{2} , \ \la \gg 1.
\end{equation}
where $\beta \in C^{\infty}_{0}((\frac{1}{2}, 2))$ is a Littlewood-Paley bump function. By writing $[0,1]$ as a union of $\sim \la$ intervals of length $\la^{-1}$, it follows that (\ref{short time dyadic estimate}) implies
\begin{equation}
    \mixlpnorm{e^{-it\laplac{g}} \beta(P/ \la)f}{p}{q}{M \times [0, 1]} \leq C \la^{\frac{1}{p}} \altlpnorm{f}{2} , \ \la \gg 1
\end{equation}
by adding up the uniform estimates on the short time intervals. This then implies (\ref{abstract strichartz}) with $s \geq \frac{1}{p}$ using Littlewood-Paley theory. The restriction $s \geq \frac{1}{p}$ is necessary on a general compact manifold since the derivative loss of $\frac{1}{2}$ for the endpoint Strichartz estimate where $(p,q) = \big(2, \frac{2d}{d-2}\big)$ for $d \geq 3$ is sharp on the standard round sphere $S^d$. One can see this by taking the initial data in (\ref{schrodinger equation}) to be a zonal eigenfunction. 

However, on compact manifolds with other types of geometries, the restriction $s \geq \frac{1}{p}$ can be improved. For example, if we assume that $M$ has nonpositive sectional curvature, then Huang-Sogge recently showed in \cite{HuangSoggeStrichartz} that we can get a logarithmic improvement over the results of Burq-Gerard-Tzvetkov
\begin{equation} \label{Huang-Sogge Strichartz}
    \mixlpnorm{u}{p}{q}{M \times [0,1]} \lesssim \lpnorm{(I+P)^{\frac{1}{p}}\log{(2I+P)}^{-\frac{1}{p}}f}{2}{M}.
\end{equation}
Huang-Sogge obtained their result by following the main idea of Burq-Gerard-Tzvetkov and proving a uniform estimate over short time intervals. However, the estimate (\ref{short time dyadic estimate}) over an interval of length $\sim \la^{-1}$ cannot be improved on any manifold, which can be seen by taking $f(x) = \beta(P/\la)(x,x_0)$ where $\beta(P/\la)(x,y)$ is the kernel of the Littlewood-Paley operator and $x_0$ is an arbitrary point on $M$. To obtain their improvement, Huang-Sogge proved lossless estimates on longer time intervals than the ones used by Burq-Gerard-Tzvetkov. In particular, Huang-Sogge showed that
\begin{equation} \label{Huang-Sogge small time Strichartz}
    \mixlpnorm{e^{-it\laplac{g}}\beta(P / \la)f}{p}{q}{M \times [0,\frac{\log\la}{\la}]} \leq C \altlpnorm{f}{2}.
\end{equation}
This paper followed earlier work by Blair-Huang-Sogge  \cite{BlairHuangSoggeStrichartz} where slightly weaker results were proven.

When $M = \Tn$ is a rectangular torus, there has been significant improvement over the restriction $s \geq \frac{1}{p}$. In this case, there is also interest in the problem when $d = 1$ and $\T{} = S^{1}$.  Perhaps the most notable work on this problem is that of Bourgain \cite{bourgain1993strichartzI} and Bourgain-Demeter \cite{BourgainDemeter2015} who showed that
\begin{equation} \label{strichartz torus}
    \lpnorm{e^{-it\laplac{\Tn}}f}{q}{\Tn \times [0,1]} \lesssim \norm{f}_{H^s(\Tn)}, \ s > \alpha(q)
\end{equation}
where 
\begin{equation}
    \alpha(q) = \\
    \begin{cases}

    & \ \ \ \ \ \ 0 \ \ \ \ \ \  ,   \ \text{if } 2 \leq q \leq \frac{2(n+1)}{n-1} \\ 
    
    &\frac{n-1}{2} - \frac{n+1}{q}, \ \text{if }  \frac{2(n+1)}{n-1} \leq q \leq \infty.

    \end{cases} 
\end{equation}
Bourgain originally showed this when $n = 2, 3$ on the standard square torus using number theoretic methods, while Bourgain-Demeter extended the result to higher dimensions and any rectangular torus as a consequence of their powerful decoupling theorem for the paraboloid. Killip-Visan later showed in \cite{KillipVisanStrichartz} that we can take $s \geq \alpha(q)$ when $q > \frac{2(n+1)}{n-1}$.

The exponent $\frac{2(n+1)}{n-1}$ is frequently called the critical exponent or the Stein-Tomas exponent, and from this point forward we denote it by $q_c = \frac{2(n+1)}{n-1}$. Note that $(q_c,q_c)$ is the only Keel-Tao admissible pair of exponents for which $p = q$. It is conjectured that the correct estimate for the square torus at $q_c$ is 
\begin{equation} \label{conjectured stricharts for torus}
    \lpnorm{e^{-it\laplac{\Tn}} f}{q_c}{\Tn \times [0, 1]} \lesssim \big( \log \la \big)^{\frac{1}{q_c}} \altlpnorm{f}{2}, \ \la \gg 1
\end{equation}
for any function $f \in L^2(\Tn)$ such that $\supp\widehat{f} \subset B(0,C\la)$ for some constant $C > 0$. Obtaining an estimate this strong when $n \geq 4$ is wide open. However, when $n = 2$ there has been significant progress and when $n = 3$ the problem has been resolved entirely. First, note that when $n = 2$, $q_c = 6$. Guth-Maldague-Wang showed in \cite{GuthMaldagueWangImprovedDecoupling} that the estimate in (\ref{conjectured stricharts for torus}) holds with the exponent $\frac{1}{6}$ on the logarithmic factor replaced by some large constant $C$. In \cite{GuoLiYungImprovedStrichartz}, Guo-Li-Yung adapted the methods of the previous work and improved the exponent on the logarithmic factor to $2 + \varepsilon$ for any $\varepsilon > 0$. When $n = 3$, $q_c = 4$ and in the remarkable paper \cite{HerrKwak2024}, Herr-Kwak showed that if $S \subset \Z{2}$ then
\begin{equation} \label{Herr-Kwak Strichartz}
    \lpnorm{e^{-it\laplac{\Tn}}\mathbbm{1}_S(D_x)f}{4}{\T{2} \times [0,1]} \lesssim \big(\log \#S\big)^{\frac{1}{4}}\altlpnorm{f}{2}.
\end{equation} 
Herr-Kwak achieved this by proving a stronger inequality on a short time window, namely
\begin{equation} \label{Herr-Kwak short time Strichartz}
    \lpnorm{e^{-it\laplac{\Tn}}\mathbbm{1}_S(D_x)f}{4}{\T{2} \times [0,\frac{1}{\log \#S}]} \lesssim \altlpnorm{f}{2}.
\end{equation}

Returning to the more general case of rectangular tori, the decoupling theorem of Bourgain-Demeter also yields an improvement on $\Tn$ over the work of Burq-Gerard-Tzvetkov at the endpoint exponent pair $\big(2, \frac{2d}{d-2}\big)$. In particular, using H\"{o}lder's inequality, Sobolev embedding and (\ref{strichartz torus}) at $q_c$ one can show that 
\begin{equation} \label{endpoint strichartz tori}
    \mixlpnorm{e^{-it\laplac{\Tn}}f}{2}{\frac{2d}{d-2}}{\Tn \times [0,1]} \lesssim \norm{f}_{H^s(\Tn)}, \ \ \text{for any} \ s > \frac{2}{d+2}.
\end{equation}
See \cite{BlairHuangSoggeStrichartz} for more details. It is expected that the restriction $s > \frac{2}{d+2} = \frac{2}{n+1}$ can be improved upon.

In this paper, we prove lossless Stricharts estimates on rectangular tori over short time intervals at $q_c$ and at the endpoint exponent pair $\big(2,\frac{2d}{d-2}\big)$. For convenience, we set  $q_e = \frac{2(n-1)}{n-3} = \frac{2d}{d-2}$. We first state our theorem for $q_c$. Let $\psi \in C_{0}^{\infty}(\mathbb{R})$ be an even function with $\supp \psi \subset B(0,C)$ for some constant $C > 0$. 
\begin{theorem}
    Let $\varepsilon > 0$. Then if $\delta \geq \la^{-\frac{1}{n+1} + \varepsilon}$ we have
    \begin{equation}
        \lpnorm{e^{-it\laplac{\Tn}}\psi(P/\la)f}{q_c}{\Tn \times [0, \frac{1}{\la \delta}]} \lesssim_{\varepsilon} \lpnorm{f}{2}{\Tn}.
    \end{equation}
\end{theorem}
\noindent
When $\delta > \la^{-\frac{1}{n+1}}$ then $\big(\la \delta\big)^{-1} < \la^{-\frac{n}{n+1}}$. Thus, Theorem 1.1 says that we have uniform lossless estimates on any time interval with length less than $\la^{-\frac{n}{n+1}}$. We point out that Theorem 1.1 does not recover the estimate (\ref{strichartz torus}) of Bourgain-Demeter. To recover this estimate, one would need to prove uniform lossless estimates on time intervals of length $\sim \la^{-\varepsilon}$ for any $\varepsilon > 0$. This equates to having $\delta \geq \la^{-1 + \varepsilon}$ in Theorem 1.1. We hope to work further towards this goal in future projects. However, we also note that $\la^{-\frac{n}{n+1}} = \frac{\la^{\frac{1}{n+1}}}{\la} \gg \frac{\log \la}{\la}$ so Theorem 1.1 is a significant improvement over the results of Huang-Sogge (\ref{Huang-Sogge small time Strichartz}) at $q_c$, which were proved for any compact manifold with nonpositive curvature. Our results are also an improvement of sorts on work of Schippa \cite{Schippa2023Strichartz}, who showed that for $2 \leq q < q_c$ we have
\begin{equation}
    \lpnorm{e^{-it\laplac{\Tn}}\psi(P / \la)f}{q}{\Tn \times [0, \la^{-\alpha}]} \lesssim \la^{-\kappa}\altlpnorm{f}{2}
\end{equation}
for any $\alpha > 0$ and $\kappa < \alpha\big(\frac{n-1}{2}(\frac{1}{2} - \frac{1}{p}) - \frac{1}{p}\big)$. However, the argument used for $2 \leq q < q_c$ fails to work at $q_c$ so our work compliments Schippa's result.

When $\Tn$ is the square torus our result at $q_c$ is only new for $n \geq 4$. When $n = 3$ our result is beaten by the result of Herr-Kwak (\ref{Herr-Kwak short time Strichartz}). When $n = 2$ our result is beaten by work of McConnell \cite{McConnell2025ShortTimeStrichartz}, who showed that you can prove a uniform lossless estimate on time intervals with length smaller than $\la^{-\frac{131}{208}}$. McConnell showed this by connecting the estimate over a short time interval to the problem of counting lattice points in a thin annular region of $\R{2}$. However, there is still room for improvement here, and we hope to explore this problem in future work as well.  Our result is new for general rectangular tori for all $n \geq 2$.

Next, we state our lossless Strichartz estimate on rectangular tori over a short time interval at the endpoint $(2,q_e)$.
\begin{theorem}
    Let $n \geq 4$ and $\varepsilon > 0$. Then if $\delta \geq  \la^{-\frac{n-3}{(n-1)(n+3)} + \varepsilon}$ we have
    \begin{equation}
        \mixlpnorm{e^{-it\laplac{\Tn}}\psi(P/\la)f}{2}{q_e}{\Tn \times [0,\frac{1}{\la \delta}]} \lesssim_{\varepsilon} \lpnorm{f}{2}{\Tn}
    \end{equation}
\end{theorem}
\noindent
This endpoint estimate is new for $n \geq 4$ (i.e. $d \geq 3$), but $n=4$ is also the smallest space-time dimension at which the endpoint is defined. Again, this does not improve on or recover the estimate (\ref{endpoint strichartz tori}). However, when $\delta > \la^{-\frac{n-3}{(n-1)(n+3)}}$ we have $(\la \delta)^{-1} < \frac{\la^{\frac{n-3}{(n-1)(n+3)}}}{\la}$ and $\frac{\la^{\frac{n-3}{(n-1)(n+3)}}}{\la}  \gg \frac{\log\la}{\la}$ so Theorem 1.2 is also a significant improvement over the result of Huang-Sogge (\ref{Huang-Sogge small time Strichartz}).

One interesting corollary of our result is a spectral projection estimate at $q_e = \frac{2d}{d-2}$. 
\begin{corollary}
    Let $n-1 = d \geq 3$ and let $\varepsilon > 0$. Then if $\delta \geq \la^{-\frac{d-2}{d(d+4)}+\varepsilon} = \la^{-\frac{n-3}{(n-1)(n+3)} + \varepsilon}$ we have
    \begin{equation} \label{main spectral projection estimate}
        \lpnorm{P_{\la, \delta}f}{q_e}{\T{d}} \lesssim_{\varepsilon}(\la\delta)^{\frac{1}{2}}\lpnorm{f}{2}{\T{d}}.
    \end{equation}
\end{corollary}
\noindent
The operator $P_{\la,\delta}$ projects onto the eigenspaces associated to eigenvalues of $P = \sqrt{-\laplac{\Tn}}$ that lie in the interval $[\la - \delta, \la + \delta]$. The fact that a lossless Strichartz estimate at $(2,q_e)$ on a short time interval implies a spectral projection estimate at $q_e$ is well-known. For example, a proof can be found in \cite{BlairHuangSoggeStrichartz}. However, we also include a short proof at the end of the paper for the sake of completeness. It is conjectured that (\ref{main spectral projection estimate}) holds for $\delta \geq \la^{-1+\varepsilon}$ for any $\varepsilon > 0$. The problem has been studied extensively, and one can consult the works of Hickman \cite{HickmanUniformResolvent}, Germain-Myerson \cite{Germain2022} and Germain-Myerson-Pezzi \cite{GermainMyersonPezzi} for further details. Our Corollary 1.2.1 is not an improvement over the known results established in the literature. However, it is interesting to note that if one could prove a lossless Strichartz estimate at $(2,q_e)$ on time intervals with length $\leq \la^{-\varepsilon}$ then this would imply the lower bound $\delta \geq \la^{-1 + \varepsilon}$ as result. 

Finally, let us say something about the inherent value of proving lossless Strichartz estimates on short time intervals. While the most obvious use of proving such estimates is to prove estimates with some loss in the spectral parameter over unit time intervals, these short time estimates also have some additional value. First, it is interesting to see just how long solutions to the Schr\"{o}dinger equation on a compact manifold display dispersive properties similar to the ones we see in Euclidean space. Studying how long such a time interval can be before we take some loss in the estimate is a very nice way of quantifying this property. These kinds of estimates are also more useful to proving well-posedness results for certain classes of non-linear Schr\"{o}inger equations than the corresponding estimates over unit length time intervals with loss. See \cite{HerrKwak2024}, \cite{McConnell2025ShortTimeStrichartz}, and \cite{Schippa2023Strichartz} for more details.

This paper is organized as follows. We first focus on proving Theorems 1.1 and 1.2 when $\Tn$ is the square torus, and most of our attention is devoted to the proof of Theorem 1.1. In Section 2, we introduce some notation and give a brief description of the overall scheme of the proof. Sections 3-7 contain the proof of Theorem 1.1. Section 8 then details the adjustments to the proof of Theorem 1.1 needed to prove Theorem 1.2 on the square torus. Finally, in Section 9, we explain the adjustments needed to extend our results to rectangular tori, before concluding with the short proof of Corollary 1.2.1.

\section*{Acknowledgments} 
I would like to thank my advisors, Christopher Sogge and Xiaoqi Huang, for their patient support and guidance throughout the entirety of this project. I would also like to thank Daniel Pezzi for many productive conversations regarding the problems addressed in this work. Finally, I would also like to thank Zhexing Zhang and Simone Masserini for their helpful feedback and reviews of earlier drafts of this paper.

\section{Notation, Initial Reductions and a Brief Description of the Proof}
From this point forward we let $\Tn$ denote the square torus $\Tn = \big(\R{} / 2\pi\Z{}\big)^{n-1}$. For any set $E \subset \Rn$ we let $\tilde{E} = E \cap \Zn$. Let $\la \in 2^{\N}$ with $\la \gg 1$, and $\delta \in 2^{-\N{}}$ with $\delta > \la^{-1}$. Also, let $Q_{\la} = [-\la, \la]^{n-1}$. We write $A \lesssim B$ to mean that there is a constant $C \geq 1$ depending only on the dimension such that $A \leq CB$. Similarly, $A \lesssim_{\varepsilon} B$ means there exists a constant $C_{\varepsilon}$ depending only on the dimension and $\varepsilon > 0$ such that $A \leq C_{\varepsilon}B$. We write $A \sim B$ if there are constants $C_1$ and $C_2$ such that $A \leq C_1 B$ and $B \leq C_2 A$. The value of such constants might change from line to line. This is unimportant, since these constants will not depend on the special parameters $\la$ and $\delta$. We write $e(t) = e^{it}$. For any function $f \in L^{2}(\Tn)$ and any $k \in \Zn$ we define 
\begin{equation}
    \widehat{f}(k) = \int_{\Tn}f(y)e(-y\cdot k)dy.
\end{equation}
Also, if $m$ is any function of $n-1$ variables then we define the Fourier multiplier associated to $m$ as
\begin{equation}
    m(D_x)f(x) = \sum_{k \in \Zn}\widehat{f}(k)m(k)e(x \cdot k).
\end{equation}
For a function $h \in \mathcal{S}(\R{m})$ we denote the Fourier transform of $h$
\begin{equation}
    \widehat{h}(\xi) = \int_{\R{m}}h(y)e(-y\cdot\xi)dy
\end{equation}
with inverse Fourier transform
\begin{equation}
    \check{h}(x) = \frac{1}{(2\pi)^m}\int_{\R{m}}h(\xi)e(x\cdot\xi)d\xi.
\end{equation}

Consider a function $f(x) = \sum_{k \in \tilde{Q}_{\la}}\widehat{f}(k)e(x\cdot k)$. Let $\phi \in C_{0}^{\infty}((-2,2))$ be an even function such that $\phi \equiv 1$ on $[-1,1]$. Set
\begin{equation}
    \beta(\xi) = \prod_{i = 1}^{n-1}\phi(\xi_i).
\end{equation}
Note that $\beta(\cdot / \la) \equiv 1$ on the cube $[-\la, \la]^{n-1}$ and $\supp{\beta(\cdot / \la)} \subset [-2\la, 2\la]^{n-1}$. Also, $\beta(D_x/\la)f = f$ since $\beta(\cdot / \la) \equiv 1$ on $Q_{\la}$. We will prove that for any $\varepsilon > 0$ if $\delta \geq \la^{-\frac{1}{n+1} +\varepsilon}$ then
\begin{equation}
    \lpnorm{e^{-it\laplac{\Tn}}\beta(D_x/\la)f}{q_c}{\Tn \times [0,\frac{1}{\la\delta}]} \lesssim_{\varepsilon}\lpnorm{f}{2}{\Tn}.
\end{equation}
Note that this estimate implies Theorem 1.1 in the case of the square torus via a straightforward argument. We also note that for any $\delta > \la^{-1}$ the estimate

\begin{equation} \label{dilated strichartz without cutoff}
    \lpnorm{e^{-i\la^{-1}t\laplac{\Tn}}\beta(D_x / \la) f}{q_c}{\Tn \times [0, \delta^{-1}]} \lesssim \la^{\frac{1}{q_c}}\altlpnorm{f}{2}
\end{equation}
is equivalent to the estimate
\begin{equation}
    \lpnorm{e^{-it\laplac{\Tn}}\beta(D_x / \la) f}{q_c}{\Tn \times [0, \frac{1}{\la \delta}]} \lesssim \altlpnorm{f}{2}
\end{equation}
by a simple rescaling argument.

Now, take a bump function $\eta \in C^{\infty}_{0}((-1,1))$ with $\eta \equiv 1$ on $[-\frac{1}{2}, \frac{1}{2}]$. We define the time-dilated Schr\"{o}dinger operator
\begin{equation} \label{time-dilated schordinger operator}
    \Sl f(x,t) = \eta(\delta t)e^{-i\la^{-1}t\laplac{\Tn}}\beta(D_x / \la)f(x).
\end{equation}
This operator takes functions of $n-1$ variables to functions of $n$ variables. Sometimes, we might want to think of this operator as one defined in terms of an additional parameter $t \in [0,\delta^{-1}]$ that maps functions of $n-1$ variables to functions of $n-1$ variables. When we wish to emphasize this perspective, we will write $S_{\la ,t}f(x) = \Sl f(x,t)$. It is straightforward to see that (\ref{dilated strichartz without cutoff}) follows from the following estimate for the time-dilated Schr\"{o}dinger operator

\begin{equation} \label{dilated strichartz with cutoff}
    \lpnorm{\Sl f}{q_c}{\Tn \times [0, \delta^{-1}]} \lesssim \la^{\frac{1}{q_c}}\altlpnorm{f}{2}.
\end{equation}
Note that for the functions $f$ described above we have

\begin{equation}
    \Sl f(x,t) = \eta(\delta t)\sum_{k \in \tilde{Q_{\la}}}\widehat{f}(k)\beta(k / \la)e(x \cdot k + \la^{-1}t|k|^2) = \eta(\delta t)\sum_{k \in \tilde{Q_{\la}}}\widehat{f}(k)e(x \cdot k + \la^{-1}t|k|^2)
\end{equation}
since $\beta(\cdot / \la) \equiv 1$ on $Q_{\la}$. We keep the cutoff $\beta(\cdot / \la)$ in the definition as it will be useful later in Section 6 when we prove estimates for the kernel of this operator.

The proof of (\ref{dilated strichartz with cutoff}) uses a multiscale argument that is similar in many ways to the induction on scales methods used prevalently in modern harmonic analysis. More precisely, the use of many scales is similar to the arguments in \cite{FuGuthMaldagueSuperlevelSetDecoupling} and \cite{GuthMaldagueWangImprovedDecoupling}. One key feature of our argument that sets it apart from the arguments in \cite{FuGuthMaldagueSuperlevelSetDecoupling}, \cite{GuthMaldagueWangImprovedDecoupling} and, more generally, typical induction on scales arguments, is that the number of scales we need to analyze depends only on $\varepsilon$, and crucially does not depend on the parameters $\la$ or $\delta$. This means that we can essentially focus on proving the analog of Theorem 1.1 for functions whose Fourier support has been localized to a small cube and ``add up" these estimates using the natural $L^2$ orthogonality of the many different localized pieces to obtain the desired result. Although this is an oversimplification, this is one of the key insights behind our argument. The other tools used are bilinear restriction estimates for the paraboloid and a ``height-splitting" at each scale. Height-splitting was also used by Blair-Huang-Sogge in \cite{BlairHuangSoggeImprovedSpectral}, \cite{BlairHuangSoggeStrichartz} and Huang-Sogge in \cite{HuangSoggeStrichartz}, \cite{huang2024curvature} at a single scale in order to prove the results in those papers. The height-splitting technique also bears some resemblance to the pruning process for wave packets used in \cite{FuGuthMaldagueSuperlevelSetDecoupling} and \cite{GuthMaldagueWangImprovedDecoupling}.

We start by defining a small scale $\dK{} = \delta^{\frac{1}{K}}$. For each $1 \leq \ell \leq K$ we partition the cube $Q_{\la}$ into cubes $\{\tau_{\ell}\}$ of side-length $\sim \la \dK{\ell}$. For each $\ell$, we localize the Fourier support of our function to one of the cubes in the partition $\{\tau_{\ell}\}$. We then prove an estimate (see Proposition 3.1) relating the analogous estimate to (\ref{dilated strichartz with cutoff}) for our localized function to the same kind of estimate but now with our function frequency localized to a cube $\tau_{\ell + 1}$ contained in $\tau_{\ell}$. Ultimately, when we reach the scale $\dK{K} = \delta$ and localize in frequency to cubes $\tau_{K}$ the amount of constructive interference we see is minimal and we obtain the best possible estimate on a time window of size $[0, \delta^{-1}]$: 
\begin{equation}\label{small scale Strichartz}
    \lpnorm{\Sl f_{\tau_{K}}}{q_c}{\Tn \times [0, \delta^{-1}]} \lesssim \la^{\frac{1}{q_c}}\altlpnorm{f_{\tau_{K}}}{2}.
\end{equation}
It is interesting to note that (\ref{small scale Strichartz}) holds for any $\delta \geq \la^{-1 + \kappa}$ for any $\kappa > 0$. The restriction on the size of $\delta$ comes from the estimates we obtain at scales where $1 \leq \ell < K$. We are then able to use the estimates at the smallest scale to prove estimates at larger scales, and ultimately to prove (\ref{dilated strichartz with cutoff}).

\section{Height Splitting at Scale $\dK{\ell - 1}$, for $1 \leq \ell < K$}

Fix $\ell$ with $1 \leq \ell < K$. Let $\tau_{\ell - 1}$ be a cube of side-length $2\la \dK{\ell - 1}$ with sides parallel to the coordinate axes contained in $Q_{\la}$. Let $f_{\tau_{\ell -1}}(x) = \sum_{k \in \tilde{\tau}_{\ell - 1}}\widehat{f}(k)e(x \cdot k)$ so that $\fsupp{f_{\tau_{\ell-1}}} \subset \tau_{\ell - 1}$. When $\ell = 1$ we take $\tau_{0} = Q_{\la}$ and $f_{\tau_{0}} = f$. Now, let $\{\tau_{\ell}\}$ be a collection of closed cubes of side-length $2\la \dK{\ell}$ with sides parallel to the coordinate axes whose interiors are pairwise disjoint and such that $\tau_{\ell - 1} = \cup \tau_{\ell}$. For any $\tau_{\ell}$ we write $f_{\tau_{\ell}}(x) = \sum_{k \in \tilde{\tau_{\ell}}}\widehat{f}(k)e(x \cdot k)$.
We will prove the following estimate:
\begin{proposition}
Let $\kappa > 0$. Then for any $\delta \geq \la^{-1 + \kappa}$ we have
\begin{equation}
\begin{split}
    \lpnorm{\Sl f_{\tau_{\ell-1}}}{q_c}{\Tn\times[0, \delta^{-1}]} \leq \big(C_1 \la^{\frac{1}{q_c}} + C_{\varepsilon}\la^{\frac{n-1}{n}\frac{1}{q_c} + \varepsilon}\dK{-C}\delta^{-\frac{n+1}{n}\frac{1}{q_c}}\big)\altlpnorm{f_{\tau_{\ell-1}}}{2} \\
    + C_2\Big(\sum_{\tau_{\ell} \subset \tau_{\ell - 1}}\lpnorm{\Sl f_{\tau_{\ell}}}{q_c}{\Tn \times [0, \delta^{-1}]}^{2}\Big)^{\frac{1}{2}}
\end{split}
\end{equation}
for any $\varepsilon > 0$.
\end{proposition}
\noindent Note that $C, C_1, C_2$ and $C_{\varepsilon}$ are absolute constants which all depend on the dimension $n$ with $C_{\varepsilon}$ also depending on $\varepsilon$. Crucially, none of these constants depend on either of the parameters $\la$ or $\delta$. It suffices to prove Proposition 3.1 when $\altlpnorm{f_{\tau_{\ell - 1}}}{2} = 1$ so we will assume that this is the case until the end of Section 5.

We define the following two subsets of $\Tn \times [0,\delta^{-1}]$:
\begin{equation} \label{height decomposition definitions}
\begin{split}
    A^{+}_{\tau_{\ell - 1}} = \{(x,t) \in \Tn \times [0,\delta^{-1}] \ : |\Sl f_{\tau_{\ell - 1}}(x,t)| \geq C_0\big(\la \delta^{-1}\dK{2(\ell - 1)}\big)^{\frac{n-1}{4}} \} \\
    A^{-}_{\tau_{\ell - 1}} = \{(x,t) \in \Tn \times [0,\delta^{-1}] \ : |\Sl f_{\tau_{\ell - 1}}(x,t)| < C_0\big(\la \delta^{-1}\dK{2(\ell - 1)}\big)^{\frac{n-1}{4}} \}
\end{split}
\end{equation}

\noindent
where the precise value of $C_0$ will be chosen later in the argument. Then Proposition 3.1 will follow from estimates for $\Sl f_{\tau_{\ell - 1}}$ over each of the regions $A^{+}_{\tau_{\ell - 1}}$ and $A^{-}_{\tau_{\ell - 1}}$, which we describe in the following two propositions.

\begin{proposition}
Let $\kappa  > 0$. Then for any $\delta \geq \la^{-1 + \kappa}$ we have
\begin{equation}
    \lpnorm{\Sl f_{\tau_{\ell - 1}}}{q_c}{A^{+}_{\tau_{\ell - 1}}} \lesssim \la^{\frac{1}{q_c}}\altlpnorm{f_{\tau_{\ell - 1}}}{2}.
\end{equation}
\end{proposition}

\begin{proposition}
Let $\kappa > 0$. Then for any $\delta \geq \la^{-1 + \kappa}$ we have
    \begin{equation}
        \lpnorm{\Sl f_{\tau_{\ell - 1}}}{q_c}{A^{-}_{\tau_{\ell  -1}}} \leq C\Big(\sum_{\tau_{\ell} \subset \tau_{\ell - 1}}\lpnorm{\Sl f_{\tau_{\ell}}}{q_c}{\Tn \times [0, \delta^{-1}]}^{2}\Big)^{\frac{1}{2}} + C_{\varepsilon}\la^{\frac{n-1}{n}\frac{1}{q_c}  + \varepsilon}\dK{-C}\delta^{-\frac{n+1}{n}\frac{1}{q_c}}\altlpnorm{f_{\tau_{\ell - 1}}}{2}
    \end{equation}
for any $\varepsilon > 0$.
\end{proposition}

\section{Large Height Estimates at Scale $\dK{\ell-1}$, for $1 \leq \ell < K$}

We start by proving Proposition 3.2. To do this, we adapt an argument used in recent works of Blair-Huang-Sogge \cite{BlairHuangSoggeImprovedSpectral}, \cite{BlairHuangSoggeStrichartz} and Huang-Sogge \cite{HuangSoggeStrichartz}, \cite{huang2024curvature}. First, note that if the center of $\tau_{\ell - 1}$ is $c = (c_1, ..., c_{n-1})$ then $\tau_{\ell - 1} = \prod_{i=1}^{n-1}[c_{i} - \la\dK{\ell - 1}, c_{i} + \la\dK{\ell - 1}]$. Let $\chi \in C_{0}^{\infty}((-2,2))$ with $\chi \equiv 1$ on $[-1,1]$. For $\xi = (\xi_1,...,\xi_{n-1})$ set $\chi_{\tau_{\ell-1}}(\xi) =\prod_{i = 1}^{d} \chi((\la \dK{\ell - 1})^{-1}(\xi_i - c_{i}))$. Then $\supp\chi_{\tau_{\ell - 1}} \subset 2\tau_{\ell - 1}$ and $\chi_{\tau_{\ell - 1}} \equiv 1$ on $\tau_{\ell - 1}$. Thus, if $\chi_{\tau_{\ell - 1}}(D_x)$ denotes the Fourier multiplier given by
\begin{equation}
    \chi_{\tau_{\ell - 1}}(D_x)h(x) = \sum_{k \in \Zn}\widehat{h}(k)\chi_{\tau_{\ell}}(k)e(x \cdot k), \ x \in \Tn
\end{equation}
then it follows that $\Sl f_{\tau_{\ell - 1}} = \Sl \chi_{\tau_{\ell - 1}}(D_x) f_{\tau_{\ell - 1}}$. 

Now, choose a function $g$ such that
\begin{equation} \label{large height duality function}
\begin{split}
        \lpnorm{g}{q_c'}{\Tn \times [0, \delta^{-1}]} &= 1 \ \text{and} \\
        \lpnorm{\Sl \chi_{\tau_{\ell - 1}}(D_x)f_{\tau_{\ell - 1}}}{q_c}{A^{+}_{\tau_{\ell - 1}}} &= \Big|\iint \Sl \chi_{\tau_{\ell - 1}}(D_x) f_{\tau_{\ell - 1}}(x,t)\mathbbm{1}_{A^{+}_{\tau_{\ell - 1}}}(x,t)\overline{ g(x,t)} dx dt\Big|.
\end{split}
\end{equation}
As we are assuming that $\lpn{f_{\tau_{\ell - 1}}}{2} = 1$, it follows from the Cauchy-Schwarz inequality that
\begin{equation}
\begin{split}
\lpnorm{\Sl  f_{\tau_{\ell - 1}}}{q_c}{A^{+}_{\tau_{\ell - 1}}}^2 &= \lpnorm{\Sl \chi_{\tau_{\ell - 1}}(D_x) f_{\tau_{\ell - 1}}}{q_c}{A^{+}_{\tau_{\ell - 1}}}^2 \\ 
    &= \Big|\int f_{\tau_{\ell - 1}}(x)\overline{\big(\Sl \chi_{\tau_{\ell - 1}}(D_x)\big)^* (\mathbbm{1}_{A^{+}_{\tau_{\ell - 1}}}g)(x)} dx \Big|^2 \\
     &\leq \int \big|\big(\Sl \chi_{\tau_{\ell - 1}}(D_x)\big)^* (\mathbbm{1}_{A^{+}_{\tau_{\ell - 1}}}g)(x) \big|^2 dx \\
    &= \iint \big(\Sl \chi_{\tau_{\ell - 1}}(D_x)\big)^* (\mathbbm{1}_{A^{+}_{\tau_{\ell - 1}}}g)(x) \overline{\big(\Sl \chi_{\tau_{\ell - 1}}(D_x)\big)^* (\mathbbm{1}_{A^{+}_{\tau_{\ell - 1}}}g)(x)} dx \\
     &= \iint \Sl \chi_{\tau_{\ell -1}}(D_x) \big( \Sl \chi_{\tau_{\ell - 1}}(D_x)\big)^*\big( \mathbbm{1}_{A^{+}_{\tau_{\ell - 1}}}g\big)(x,t) \overline{\big( \mathbbm{1}_{A^{+}_{\tau_{\ell - 1}}}g\big)(x,t)}dxdt \\
    &=\iint L_{\tau_{\ell - 1}}^{\la}\big( \mathbbm{1}_{A^{+}_{\tau-{\ell - 1}}}g\big)(x,t) \overline{\big( \mathbbm{1}_{A^{+}_{\tau_{\ell - 1}}}g\big)(x,t)}dxdt \\ 
    &+ \iint G_{\tau_{\ell - 1}}^{\la}\big( \mathbbm{1}_{A^{+}_{\tau_{\ell - 1}}}g\big)(x,t) \overline{\big( \mathbbm{1}_{A^{+}_{\tau_{\ell - 1}}}g\big)(x,t)}dxdt \\ 
    &= I + II
\end{split}
\end{equation}
where $L_{\tau_{\ell -1}}^{\la}$ is the operator with kernel equaling that of $\Sl \chi_{\tau_{\ell - 1}}(D_x)\big(\Sl \chi_{\tau_{\ell - 1}}(D_x)\big)^*$ if $|t-s| \leq 1$ and $0$ otherwise, i.e. 
\begin{equation} \label{large height local kernel I}
    L_{\tau_{\ell - 1}}^{\la}(x,t;y,s) = \\
    { \begin{cases} 

    &\Sl \chi_{\tau_{\ell - 1}}(D_x)\big(\Sl \chi_{\tau_{\ell - 1}}(D_x)\big)^{*}(x,t;y,s), \ \text{if } |t-s| \leq 1 \\ 
    
    & \hfill 0 \hfill, \ \ \ \ \ \text{otherwise} 

    \end{cases} 
}
\end{equation}
where
\begin{equation} \label{large height local kernel II}
    \Sl \chi_{\tau_{\ell - 1}}(D_x)\big(\Sl \chi_{\tau_{\ell - 1}}(D_x)\big)^{*}(x,t;y,s) = \eta(\delta t)\eta(\delta s)\big(\chi_{\tau_{\ell - 1}}^2(D_x)\beta^2(D_x/\la)e^{-i\la^{-1}(t-s)\laplac{\Tn}}\big)(x,y).
\end{equation}
It is straightforward to see that the short time Strichartz estimate of Burq-Gerard-Tzvetkov (\ref{short time dyadic estimate}) from \cite{bgtstrichartzmanifold} implies that
\begin{equation}
    \Lpopnormalt{L_{\tau_{\ell -1}}^{\la}}{q_c'}{x,t}{q_c}{x,t} = O\big(\la^{\frac{2}{q_c}}\big).
\end{equation}
Then we can use this and H\"{o}lder's inequality to see that 
\begin{equation} \label{large height 1st term estimate}
        \big| I \big| \leq \lpn{L_{\tau_{\ell - 1}}^{\la}\big( \mathbbm{1}_{A^{+}_{\tau_{\ell - 1}}}g\big)}{q_c} \cdot \lpn{\mathbbm{1}_{A^{+}_{\tau_{\ell - 1}}}g}{q_c'} \lesssim \la^{\frac{2}{q_c}} \lpn{\mathbbm{1}_{A^{+}_{\tau_{\ell - 1}}}g}{q_c'}^2 \leq \la^{\frac{2}{q_c}}\lpn{g}{q_c'}^2 = \la^{\frac{2}{q_c}}.
\end{equation}
In Section 6, Corollary 6.0.1, we will show that 
\begin{equation} \label{kernel estimate for large heights}
\begin{split}
    \big| \Sl \chi_{\tau_{\ell - 1}}(D_x) \big(\Sl \chi_{\tau_{\ell - 1}}(D_x)\big)^*(x,t;y,s) \big| &\lesssim  \la^{\frac{n-1}{2}}|t-s|^{-\frac{n-1}{2}}(1+\dK{\ell-1}|t-s|)^{n-1}\\
    & \lesssim \delta_{K}^{(n-1)(\ell-1)} \big(\la \delta^{-1}\big)^{\frac{n-1}{2}}, \ \text{if } |t-s| \leq 2\delta^{-1}.
\end{split}
\end{equation}
It follows that 
\begin{equation}
    \Lpopnormalt{G_{\tau_{\ell - 1}}^{\la}}{1}{x,t}{\infty}{x,t} \leq C\delta_{K}^{(n-1)(\ell - 1)}(\la \delta^{-1})^{\frac{n-1}{2}}
\end{equation}
for some constant $C$. Combining this with H\"{o}lder's inequality we obtain
\begin{equation}
\begin{split}
    \big| II \big| &\leq C\delta_{K}^{(n-1)\ell}(\la \delta^{-1})^{\frac{n-1}{2}} \lpn{\mathbbm{1}_{A^{+}_{\tau_{\ell - 1}}}g}{1}^2 \leq C\delta_{K}^{(n-1)(\ell-1)}(\la \delta^{-1})^{\frac{n-1}{2}}\lpn{g}{q_c'}^{2}\cdot \lpn{\mathbbm{1}_{A^{+}_{\tau_{\ell - 1}}}}{q_c}^2 \\
    &= C\delta_{K}^{(n-1)(\ell-1)}(\la \delta^{-1})^{\frac{n-1}{2}}\lpn{\mathbbm{1}_{A^{+}_{\tau_{\ell - 1}}}}{q_c}^2.
\end{split}
\end{equation}
Recalling the definition of the set $A^{+}_{\tau_{\ell - 1}}$ in (\ref{height decomposition definitions}), we see that
\begin{equation}
\begin{split}
    \lpn{\mathbbm{1}_{A^{+}_{\tau_{\ell - 1}}}}{q_c}^2 &\leq \Big(C_0\big(\la \delta^{-1}\delta_{K}^{2(\ell - 1)}\big)^{\frac{n-1}{4}}\Big)^{-2} \lpnorm{\Sl f_{\tau_{\ell - 1}}}{q_c}{A^{+}_{\tau_{\ell - 1}}}^2 \\
    &= \Big(C_0^2 (\la \delta^{-1})^{\frac{n-1}{2}}\delta_{K}^{(n-1)(\ell - 1)}\Big)^{-1}\lpnorm{\Sl f_{\tau_{\ell - 1}}}{q_c}{A^{+}_{\tau-{\ell - 1}}}^2.
\end{split}
\end{equation}
Choosing $C_0$ large enough so that $C_0^{-2}C \leq \frac{1}{2}$ yields 
\begin{equation}
    \big| II \big| \leq \frac{1}{2}\lpnorm{\Sl f_{\tau_{\ell - 1}}}{q_c}{A^{+}_{\tau_{\ell - 1}}}.
\end{equation}
Combining this with our estimate for $I$ in (\ref{large height 1st term estimate}), we have
\begin{equation}
    \lpnorm{\Sl f_{\tau_{\ell - 1}}}{q_c}{A^{+}_{\tau_{\ell - 1}}}^2 \leq C \la^{\frac{2}{q_c}} + \frac{1}{2}\lpnorm{\Sl f_{\tau_{\ell  -1}}}{q_c}{A^{+}_{\tau_{\ell - 1}}}^2
\end{equation}
 which completes the proof of Proposition 3.2.

 \section{Small Height Estimates at Scale $\dK{\ell-1}$, for $1 \leq \ell < K$}

We now turn our attention to proving Proposition 3.3. To do so, we will use an elementary variant of the Broad-Narrow argument originally developed by Bourgain-Guth in \cite{BourgainGuthMultilinear}. The decomposition contained in the following lemma is similar to Lemma 7.2 in \cite{demeter2020fourier}.

\begin{lemma}
    There are constants $C, C_1, C_2$, independent of $\la$ and $\delta$, such that 
    \begin{equation}
        \Big|\sum_{\tau_{\ell} \subset \tau_{\ell - 1}}\Sl f_{\tau_{\ell}}(x,t)\Big| \leq C_1 \max_{\tau_{\ell} \subset \tau_{\ell - 1}}|\Sl f_{\tau_{\ell}}(x,t) | + C_2 \dK{-C}\max_{\substack{\tau_{\ell},\tau_{\ell}' \subset \tau_{\ell - 1} \\ dist(\tau_{\ell}, \tau_{\ell}') \gtrsim \la \dK{\ell}}}|\Sl f_{\tau_{\ell}}(x,t)\Sl f_{\tau_{\ell}'}(x,t)|^{\frac{1}{2}}.
    \end{equation}
\end{lemma}

\begin{proof}
Fix a point $(x,t) \in \Tn \times [0,\delta^{-1}]$. Let $\tau_{\ell}^{*} \subset \tau_{\ell - 1}$ be such that 
\begin{equation}
    |\Sl f_{\tau_{\ell}^{*}}(x,t)| = \max_{\tau_{\ell} \subset {\tau_{\ell - 1}}} |\Sl f_{\tau_{\ell}}(x,t)|.
\end{equation} 
Next, define the set
\begin{equation}
     S_{\tau_{\ell-1}}^{\text{big}} = \{\tau_{\ell} \subset \tau_{\ell - 1} \ : \ |\Sl f_{\tau_{\ell}}(x,t)| \geq \dK{n-1} |\Sl f_{\tau_{\ell}^{*}}(x,t)|\}.   
\end{equation}
First, suppose that there is some $\tau_{\ell}' \in S_{\tau_{\ell - 1}}^{\text{big}}$ such that $dist(\tau_{\ell}^{*}, \tau_{\ell}') \gtrsim \la \dK{\ell}$. Since $\#\{\tau_{\ell}\subset \tau_{\ell - 1}\} \lesssim \dK{-(n-1)}$ we have
\begin{equation}
    \Big|\sum_{\tau_{\ell} \subset \tau_{\ell - 1}} \Sl f_{\tau_{\ell}}(x,t)\Big| \lesssim \dK{-(n-1)}|\Sl f_{\tau_{\ell}^{*}}(x,t)| \leq \dK{-\frac{3(n-1)}{2}}| \Sl f_{\tau_{\ell}^{*}}(x,t) \Sl f_{\tau_{\ell}'}(x,t)|^{\frac{1}{2}}. 
\end{equation}
On the other hand, if every $\tau_{\ell} \in S_{\tau_{\ell - 1}}^{\text{big}}$ satisfies $dist(\tau_{\ell}, \tau_{\ell}^{*}) \lesssim \la \dK{\ell}$ then $\# S_{\tau_{\ell - 1}}^{\text{big}} = O(1)$. Thus, we have
\begin{equation}
\begin{split}
    \Big| \sum_{\tau_{\ell} \subset \tau_{\ell - 1}} \Sl f_{\tau_{\ell}}(x,t)\Big| &\leq \Big| \sum_{\tau_{\ell} \in S_{\tau_{\ell - 1}}^{\text{big}}} \Sl f_{\tau_{\ell}}(x,t)\Big| + \Big| \sum_{\tau_{\ell} \notin S_{\tau_{\ell - 1}}^{\text{big}}} \Sl f_{\tau_{\ell}}(x,t)\Big| \\
    & \leq \#S_{\tau_{\ell - 1}}^{\text{big}} |\Sl f_{\tau_{\ell}^{*}}(x,t)| + \#\{\tau_{\ell} \subset \tau_{\ell - 1}\} \dK{n-1}|\Sl f_{\tau_{\ell}^{*}}(x,t)| \\
    &\lesssim |\Sl f_{\tau_{\ell}^{*}}(x,t)|.
\end{split}
\end{equation}
\end{proof}
\noindent
\begin{remark}
    Note that the proof shows that we can take $C$ in the exponent of $\dK{-1}$ to be $C = \frac{3(n-1)}{2}$ but this is unimportant.
\end{remark}

We can replace the maximums in Lemma 5.1 with an $\ell^2$ norm and an $\ell^{q}$ norm over the collections $\{\tau_{\ell} \subset \tau_{\ell - 1}\}$ to attain a uniform estimate for all $(x,t) \in \Tn\times[0, \delta^{-1}]$:
\begin{equation} \label{broad-narrow decomposition}
\begin{split}
    \Big|\sum_{\tau_{\ell} \subset \tau_{\ell - 1}}\Sl f_{\tau_{\ell}}(x,t)\Big| &\leq C_1 \Big( \sum_{\tau_{\ell} \subset \tau_{\ell - 1}}|\Sl f_{\tau_{\ell}}(x,t) |^2 \Big)^{\frac{1}{2}} \\
    &+ C_2 \dK{-C}\Big(\sum_{\substack{\tau_{\ell},\tau_{\ell}' \subset \tau_{\ell - 1} \\ dist(\tau_{\ell}, \tau_{\ell}') \gtrsim \la \dK{\ell}}}|\Sl f_{\tau_{\ell}}(x,t)\Sl f_{\tau_{\ell}'}(x,t)|^{\frac{q}{2}}\Big)^{\frac{1}{q}}.
\end{split} 
\end{equation}
We almost have all of the tools we need to prove Proposition 3.3. The next lemma is a consequence of Tao's bilinear restriction estimate for the paraboloid from \cite{taobilinear}.

\begin{lemma} If $\tau_{\ell}, \tau_{\ell}' \subset \tau_{\ell - 1}$ with $dist(\tau_{\ell}, \tau_{\ell}') \gtrsim \la \dK{\ell}$ then for $q = \frac{2(n+2)}{n}$ we have
\begin{equation}
    \lpnorm{\big|\Sl f_{\tau_{\ell}}\Sl f_{\tau_{\ell}'}\big|^{\frac{1}{2}}}{q}{\Tn\times[0,\delta^{-1}]} \lesssim_{\varepsilon} \la^{\frac{n-1}{2} -\frac{n}{q} + \varepsilon }\dK{\ell(\frac{n-1}{2} - \frac{n+1}{q})} \delta^{-\frac{1}{q}} \Big(\altlpnorm{f_{\tau_{\ell}}}{2}\altlpnorm{f_{\tau_{\ell}'}}{2}\Big)^{\frac{1}{2}}.
\end{equation}
\end{lemma}
We postpone the proof of Lemma 5.2 for now and instead show how we can use it to prove Proposition 3.3. To start, fix $q = \frac{2(n+2)}{n}$. Applying (\ref{broad-narrow decomposition}) we have

\begin{equation} \label{estimate in terms of broad and narrow terms}
\begin{split}
    &\lpnorm{\Sl f_{\tau_{\ell - 1}}}{q_c}{A_{\tau_{\ell - 1}}^{-}}^{q_c} = \iint_{A_{\tau_{\ell - 1}}^{-}}|\Sl f_{\tau_{\ell - 1}}(x,t)|^{q_c - q}|\Sl f_{\tau_{\ell - 1}}(x,t)|^{q}dxdt \\
    &=  \iint_{A_{\tau_{\ell - 1}}^{-}}|\Sl f_{\tau_{\ell - 1}}(x,t)|^{q_c - q}\Big|\sum_{\tau_{\ell} \subset \tau_{\ell - 1}}\Sl f_{\tau_{\ell}}(x,t)\Big|^{q}dxdt \\ 
    &\leq C_1^{q}\iint_{A_{\tau_{\ell - 1}}^{-}}|\Sl f_{\tau_{\ell - 1}}(x,t)|^{q_c - q}\Big(\sum_{\tau_{\ell} \subset \tau_{\ell - 1}}|\Sl f_{\tau_{\ell}}(x,t)|^{2}\Big)^{\frac{q}{2}}dxdt \\
    &+ C_2^{q}\dK{-Cq} \sum_{\substack{\tau_{\ell},\tau_{\ell}' \subset \tau_{\ell - 1} \\ dist(\tau_{\ell}, \tau_{\ell}') \gtrsim \la \dK{\ell}}} \iint_{A_{\tau_{\ell - 1}}^{-}}|\Sl f_{\tau_{\ell - 1}}(x,t)|^{q_c - q}  |\Sl f_{\tau_{\ell}}(x,t)\Sl f_{\tau_{\ell}'}(x,t)|^{\frac{q}{2}}dxdt \\ 
    & = I + II.
\end{split}
\end{equation}
First, we deal with $I$. Let $\Sq{f_{\tau_{\ell - 1}}}{}{\ell} = \Big( \sum_{\tau_{\ell} \subset \tau_{\ell - 1}} |\Sl f_{\tau_{\ell}}|^2\Big)^{\frac{1}{2}}$ be the square function associated to our partition of $\tau_{\ell - 1}$. Then 
\begin{equation} \label{narrow term estimate}
\begin{split}
    I &= C_1^{q} \iint_{A_{\tau_{\ell - 1}}^{-}}|\Sl f_{\tau_{\ell - 1}}(x,t)|^{q_c - q} |\Sq{f_{\tau_{\ell - 1}}}{}{\ell}(x,t)|^qdxdt \\
    &\leq C_1^{q}\lpnorm{\Sl f_{\tau_{\ell - 1}}}{q_c}{A_{\tau_{\ell - 1}}^{-}}^{q_c-q} \lpnorm{\Sq{f_{\tau_{\ell - 1}}}{}{\ell}}{q_c}{\Tn \times [0, \delta^{-1}]}^{q} \\
    &\leq \frac{1}{2}\lpnorm{\Sl f_{\tau_{\ell - 1}}}{q_c}{A_{\tau_{\ell - 1}}^{-}}^{q_c} + C\lpnorm{\Sq{f_{\tau_{\ell - 1}}}{}{\ell}}{q_c}{\Tn \times [0, \delta^{-1}]}^{q_c} \\
    &\leq \frac{1}{2}\lpnorm{\Sl f_{\tau_{\ell - 1}}}{q_c}{A_{\tau_{\ell - 1}}^{-}}^{q_c} + C\Big(\sum_{\tau_{\ell} \subset \tau_{\ell - 1}} \lpnorm{\Sl f_{\tau_{\ell}}}{q_c}{\Tn \times [0, \delta^{-1}]}\Big)^{\frac{q_c}{2}}.
\end{split} 
\end{equation}
where we used H\"{o}lder's inequality in the second line, then Young's inequality (with $\epsilon = \frac{1}{2 C_{1}^{q}}$, see Appendix B of \cite{EvansPDE} for more details) in the second-to-last line and, finally, Minkowski's inequality in the last line.

Now, we need to handle $II$. Note that by definition of the set $A_{\tau_{\ell - 1}}^{-}$ in (\ref{height decomposition definitions}) and Lemma 5.2 we have
\begin{equation} \label{broad term estimate}
\begin{split}
    II &= C_2^{q}\dK{-Cq}\sum_{\substack{\tau_{\ell},\tau_{\ell}' \subset \tau_{\ell - 1} \\ dist(\tau_{\ell}, \tau_{\ell}') \gtrsim \la \dK{\ell}}} \iint_{A_{\tau_{\ell - 1}}^{-}}|\Sl f_{\tau_{\ell - 1}}(x,t)|^{q_c - q}  |\Sl f_{\tau_{\ell}}(x,t)\Sl f_{\tau_{\ell}'}(x,t)|^{\frac{q}{2}}dxdt \\
    &\lesssim \dK{-Cq} \lpnorm{\Sl f_{\tau_{\ell - 1}}}{\infty}{A_{\tau_{\ell - 1}}^{-}}^{q_c - q}\sum_{\substack{\tau_{\ell},\tau_{\ell}' \subset \tau_{\ell - 1} \\ dist(\tau_{\ell}, \tau_{\ell}') \gtrsim \la \dK{\ell}}} \iint \big| \Sl f_{\tau_{\ell}}(x,t)\Sl f_{\tau_{\ell}'}(x,t)\big|^{\frac{q}{2}}dxdt \\
    &\lesssim_{\varepsilon}  \Big(\big(\la \delta^{-1}\big)^{\frac{n-1}{4}}\dK{\frac{n-1}{2}(\ell - 1)}\Big)^{q_c - q} \la^{q\frac{n-1}{2} - n + \varepsilon} \dK{\ell(q\frac{n-1}{2} - n -1)} \delta^{-1} \dK{-Cq}\sum_{\substack{\tau_{\ell},\tau_{\ell}' \subset \tau_{\ell - 1} \\ dist(\tau_{\ell}, \tau_{\ell}) \gtrsim \la \dK{\ell}}} \Big( \altlpnorm{f_{\tau_{\ell}}}{2} \lVert f_{\tau_{\ell}'} \rVert_{2}\Big)^{\frac{q}{2}} \\
    & \leq \la^{q\frac{n-1}{2} - n + \varepsilon} \dK{\ell(q\frac{n-1}{2} - n -1)} \delta^{-1} \dK{-Cq} \Big(\big(\la \delta^{-1}\big)^{\frac{n-1}{4}}\dK{\frac{n-1}{2}(\ell - 1)}\Big)^{q_c - q} \big( \#\{\tau_{\ell} \subset \tau_{\ell - 1}\} \big)^2 \altlpnorm{f_{\tau_{\ell - 1}}}{2}^q \\
    & \lesssim \la^{q\frac{n-1}{2} - n + \varepsilon} \dK{\ell(q\frac{n-1}{2} - n -1)} \delta^{-1} \dK{-Cq} \Big(\big(\la \delta^{-1}\big)^{\frac{n-1}{4}}\dK{\frac{n-1}{2}(\ell - 1)}\Big)^{q_c - q} \dK{-2(n-1)} \altlpnorm{f_{\tau_{\ell - 1}}}{2}^q.
\end{split}
\end{equation}
Note that for the exponent on the first factor of $\dK{\ell}$ we have
\begin{equation}
    q\tfrac{n-1}{2} - n - 1= \tfrac{n-1}{2}\big(q - \tfrac{2(n+1)}{n-1}\big) = -\tfrac{n-1}{2}(q_c - q).
\end{equation}
Thus, the first factor of $\dK{\ell}$ cancels out with the second factor. Also, if we collect the exponents on the factors of $\la$ we have
\begin{equation}
    q\tfrac{n-1}{2} - n + \tfrac{n-1}{4}(q_c - q) = \tfrac{n-1}{4}(q_c + q) - n = \tfrac{n-1}{n}.
\end{equation}
Thus, (\ref{broad term estimate}) yields
\begin{equation} \label{final broad term estimate}
    II \lesssim_{\varepsilon} \la^{\frac{n-1}{n} + \varepsilon}\delta^{-\frac{n+1}{n}}\dK{-C'}\altlpnorm{f_{\tau_{\ell - 1}}}{2}^{q} = \la^{\frac{n-1}{n} + \varepsilon}\delta^{-\frac{n+1}{n}}\dK{-C'}\altlpnorm{f_{\tau_{\ell - 1}}}{2}^{q_c}.
\end{equation}
where the last equality follows from the fact that $\altlpnorm{f_{\tau_{\ell-1}}}{2} = 1$.

Combining (\ref{estimate in terms of broad and narrow terms}), (\ref{narrow term estimate}) and (\ref{final broad term estimate}) we have 
\begin{equation}
\begin{split}
        \lpnorm{\Sl f_{\tau_{\ell - 1}}}{q_c}{A_{\tau_{\ell - 1}}^{-}}^{q_c} &\leq \frac{1}{2}\lpnorm{\Sl f_{\tau_{\ell - 1}}}{q_c}{A_{\tau_{\ell - 1}}^{-}}^{q_c} + C\Big( \sum_{\tau_{\ell} \subset \tau_{\ell - 1}}\lpnorm{\Sl f_{\tau_{\ell}}}{q_c}{\Tn \times [0, \delta^{-1}]}^2\Big)^{\frac{q_c}{2}} \\
        &+ C_{\varepsilon}\la^{\frac{n-1}{n} + \varepsilon}\delta^{-\frac{n+1}{n}}\dK{-C}\altlpnorm{f_{\tau_{\ell - 1}}}{2}^{q_c}
\end{split}
\end{equation}
which yields Proposition 3.3.

\subsection{Consequences of Bilinear Restriction and the Proof of Lemma 5.2}
We now turn to the proof of Lemma 5.2. As we stated before, Lemma 5.2 will be a consequence of Tao's bilinear restriction estimate for the paraboloid from \cite{taobilinear}. Recall that the bilinear estimate yielded improved results for the linear version of the restriction problem by using the result of Tao-Vargas-Vega in \cite{TaoVargasVega1998}. Their methods inspired the variable-coefficient analog that was developed and used in the work of Blair-Huang-Sogge \cite{BlairHuangSoggeImprovedSpectral}, \cite{BlairHuangSoggeStrichartz} and Huang-Sogge \cite{HuangSoggeStrichartz}, \cite{huang2024curvature}. 

Before we state the bilinear restriction estimate, we need to introduce the mathematical objects involved in its statement. For any $f:[-1,1]^{n-1} \to \C$ define the extension operator for the paraboloid
\begin{equation}
    Ef(x,t) = \int_{[-1,1]^{n-1}} f(\xi)e(x\cdot \xi + t|\xi|^2)d\xi.
\end{equation}
Also, if $S \subset [-1,1]^{n-1}$ let 
\begin{equation}
    E_{S}f(x,t) = \int_{S} f(\xi)e(x\cdot \xi + t|\xi|^2)d\xi
\end{equation}
Now, we state the spatially localized version of the bilinear restriction estimate.
\begin{theorem}
    Let $n \geq 2$ and let $\omega_1, \omega_2$ be two cubes in $[-1,1]^{n-1}$ with $\text{dist}(\omega_1, \omega_2) > 0$. Then for each $p \geq \frac{2(n+2)}{n}$, any ball $B_R$ of radius $R \gtrsim 1$, $f:\omega_1 \cup \omega_2 \to \C$, and any $\varepsilon > 0$ we have
    \begin{equation} \label{bilinear restriction orginal}
        \lpnorm{\big|E_{\omega_1}f E_{\omega_2}f\big|^{\frac{1}{2}}}{p}{B_R} \lesssim_{\varepsilon} R^{\varepsilon}\Big( \lpnorm{f}{2}{\omega_1} \lpnorm{f}{2}{\omega_2}\Big)^{\frac{1}{2}}.
    \end{equation}
\end{theorem}
\noindent This is roughly the statement of Theorem 1.1 in \cite{taobilinear}. Note that the implicit constant in Theorem 5.3 depends on the distance between the two cubes $\omega_1$ and $\omega_2$. To better understand this dependence, we introduce some additional notation. Let $C_{R}(p, D)$ be the smallest $C > 0$ such that for any two cubes $\omega_1, \omega_2 \subset [-1,1]^{n-1}$ with $dist(\omega_1, \omega_2) \gtrsim D$ and $f: \omega_1 \cup \omega_2 \to \C$ the following inequality holds on any ball $B_R \subset \mathbb{R}^n$:
\begin{equation}
     \lpnorm{\big|E_{\omega_1}f E_{\omega_2}f\big|^{\frac{1}{2}}}{p}{B_R} \leq C\Big( \lpnorm{f}{2}{\omega_1} \lpnorm{f}{2}{\omega_2}\Big)^{\frac{1}{2}}.
\end{equation}
A parabolic rescaling argument similar to the one used in Proposition 4.2 in \cite{demeter2020fourier} gives us the following lemma.
\begin{lemma}
    Let $\omega_1$ and $\omega_2$ be two cubes in $[-1,1]^{n-1}$ of side length $ 0 < D \lesssim 1$. Assume that the distance between the centers of the two cubes is $\gtrsim D$. Then for any $1 \leq p \leq \infty$ and any $f: \omega_1 \cup \omega_2 \to \mathbb{C}$ we have
    \begin{equation}
        \lpnorm{\big|E_{\omega_1}f E_{\omega_2}f\big|^{\frac{1}{2}}}{p}{B_R} \lesssim D^{\frac{n-1}{2} - \frac{n+1}{p}} C_{RD}\Big(p, \frac{1}{2}\Big)\Big(\lpnorm{f}{2}{\omega_1} \lpnorm{f}{2}{\omega_2}\Big)^\frac{1}{2}.
    \end{equation}
\end{lemma}
\noindent 
In other words, Lemma 5.2 yields that
\begin{equation} \label{relationship between bilinear restriction constants at different scales}
    C_{R}(p, D) \lesssim D^{\frac{n-1}{2} - \frac{n+1}{p}}C_{RD}\Big(p, \frac{1}{2}\Big).
\end{equation}
If $q = \frac{2(n+2)}{n}$, then Theorem 5.1 implies that 
\begin{equation}
    C_{R}\Big(q, \frac{1}{2}\Big) \lesssim_{\varepsilon} R^{\varepsilon} \ \text{for} \ R \gtrsim 1
\end{equation}
and combining this with (\ref{relationship between bilinear restriction constants at different scales}) we have
\begin{equation} \label{estimating constant at spatial scale R and freq distance D}
    C_{R}(q, D) \lesssim_{\varepsilon} D^{\frac{n-1}{2} - \frac{n+1}{q}} R^{\varepsilon}D^{\varepsilon} \ \text{for} \ RD \gtrsim 1.
\end{equation}

Now, once again, consider two cubes $\omega_1, \omega_2$ in $[-1,1]^{n-1}$ with side length $\sim D$ and $dist(\omega_1, \omega_2) \gtrsim D $. For any $S \subset [-1,1]^{n-1}$, let $N_S(R^{-1}) = \{(\xi, |\xi|^2 + \rho) \ : \ \xi \in S, |\rho| < R^{-1} \}$. It is a straightforward exercise using a change of variables and Minkowski's integral inequality to see that we have the following estimate for any two functions $F_1, F_2$ with $\text{supp}\widehat{F_i}  \subset N_{\omega_i}(R^{-1})$, $i = 1,2$:
\begin{equation}
    \lpnorm{\big| F_1 F_2 \big|^{\frac{1}{2}}}{p}{B_R} \lesssim C_{R}(p, D) R^{-\frac{1}{2}}\Big(\lpnorm{\widehat{F_1}}{2}{N_{\omega_1}(R^{-1})}\lpnorm{\widehat{F_2}}{2}{N_{\omega_2}(R^{-1})}\Big)^{\frac{1}{2}}.
\end{equation}
Finally, this estimate implies the following bilinear exponential sum bound.
\begin{lemma}
    Let $\omega_1, \omega_2$ be two cubes in $[-1,1]^{n-1}$ with side length $\sim D > 0$ such that dist$(\omega_1, \omega_2) \gtrsim D$. Let $R \gtrsim 1$ with $R^{-1} \lesssim D$ and let $B_{R}$ be any ball of radius $R$ in $\mathbb{R}^n$. Let $\Lambda_i = \{(\xi, |\xi|^2) \in \parab{n-1} :   \xi \in \omega_i \}$, $i  =1,2$ be two collections of $R^{-1}$ separated points on the paraboloid and let $\Xi_i = \{\xi \in \omega_i : (\xi, |\xi|^2) \in \Lambda_i \}$. Then we have
    \begin{equation} \label{bilinear exponential sum estimate}
        \norm{\Big| \prod_{i = 1}^{2} \sum_{\xi \in \Xi_i}a_{\xi} e(x \cdot \xi + t|\xi|^2) \Big|^{\frac{1}{2}}}_{L_{x,t}^{p}(B_R)} \lesssim R^{\frac{n-1}{2}} C_{R}(p, D) \Big( \prod_{i = 1}^{2} \litlpnorm{a_{\xi}}{2}{\Xi_i} \Big)^{\frac{1}{2}}
    \end{equation}
\end{lemma}
\noindent 
This is essentially just Exercise 3.34 in \cite{demeter2020fourier}. Now we can combine (\ref{estimating constant at spatial scale R and freq distance D}) with (\ref{bilinear exponential sum estimate}) to conclude that if $q = \frac{2(n+2)}{n}$ we have 
\begin{equation} \label{improved bilinear exponential sum estimate}
       \norm{\Big| \prod_{i = 1}^{2} \sum_{\xi \in \Xi_i}a_{\xi} e(x \cdot \xi + t|\xi|^2) \Big|^{\frac{1}{2}}}_{L_{x,t}^{q}(B_R)} \lesssim_{\varepsilon} R^{\frac{n-1}{2} + \varepsilon} D^{\frac{n-1}{2} - \frac{n+1}{q} + \varepsilon} \Big( \prod_{i = 1}^{2} \litlpnorm{a_\xi}{2}{\Xi_i} \Big)^{\frac{1}{2}}.
\end{equation}

Define the discrete extension operator $\extensionop{}{}$ so that for any $h(x) = \sum_{k \in \tilde{Q_\la}}\widehat{h}(k)e(x \cdot k)$ we have
\begin{equation}
    \extensionop{}{h}(x,t) = \sum_{\xi \in \la^{-1}\tilde{Q_{\la}}} \widehat{h^{\la}}(\xi)e(x \cdot \xi + t|\xi|^2)
\end{equation}
where for any $\xi \in \la^{-1}\tilde{Q_{\la}}$ we have $\widehat{h^{\la}}(\xi) = \widehat{h}(\la \xi)$. We can relate this discrete extension operator to our original dilated Schr\"{o}dinger operator via the following estimate:
\begin{equation} \label{extension op vs dilated schrodinger op}
\begin{split}
    |\Sl h(x,t)|  &= \big|\eta(\delta t) \sum_{k \in \tilde{Q_{\la}}} \widehat{h}(k)e(x \cdot k + \la^{-1}t|k|^2) \big| \lesssim \big| \sum_{\xi \in \la^{-1}\tilde{Q_{\la}}}\widehat{h}(\la \xi)e(x \cdot \la \xi + t \la |\xi|^2)\big| \\
    & = \big| \sum_{\xi \in \la^{-1}\tilde{Q_{\la}}}\widehat{h^{\la}}(\xi)e(\la x \cdot \xi + \la t  |\xi|^2)\big| = |\extensionop{}{h}(\la x, \la t)|
\end{split}
\end{equation}
Let $\tau_{\ell} \subset \tau_{\ell - 1}$, one of the cubes in our partition of $\tau_{\ell - 1}$, have center $c$ so that $\tau_{\ell} = c + [- \la\dK{\ell}, \la \dK{\ell}]^{n-1}$. Then we have 
\begin{equation}
\begin{split}
        \{ \xi \in \la^{-1}\tilde{Q_{\la}} : \widehat{f_{\tau_{\ell}}^{\la}}(\xi) \neq 0 \} &\subset \{\xi \in \la^{-1}\tilde{Q_{\la}} : \la \xi \in \tau_{\ell}\} \\
        &= \{\xi \in \la^{-1}\tilde{Q_{\la}} : \xi \in \la^{-1}\tau_{\ell} \} = \la^{-1}\tilde{\tau_{\ell}}. 
\end{split}
\end{equation}
where $\la^{-1} \tau_{\ell} = \la^{-1}c + [-\dK{\ell}, \dK{\ell}]^{n-1}$. Note that $\la^{-1}\tau_{\ell} \subset  [-2,2]^{n-1} $. Thus, we have
\begin{equation}
    \extensionop{}{f_{\tau_{\ell}}}(x,t) = \sum_{\xi \in \la^{-1}\tilde{\tau_{\ell}}} \widehat{f^{\la}}(\xi)e(x \cdot \xi + t|\xi|^2).
\end{equation}

To prove Lemma 5.2 we start by noticing that if $dist(\tau_{\ell}, \tau_{\ell}') \gtrsim \la \dK{\ell}$ then $dist(\la^{-1}\tau_{\ell} , \la^{-1}\tau_{\ell}') \gtrsim \dK{\ell}$. Thus, it follows from Lemma 5.5 with $R = \la$ and $D = \dK{\ell}$ that for $q = \frac{2(n+2)}{n}$ we have
\begin{equation} \label{estimate for extension op}
\begin{split}
    \lpnorm{\big|\extensionop{}{f_{\tau_{\ell}}} \extensionop{}{f_{\tau_{\ell}'}}\big|^{\frac{1}{2}}}{q}{B_\la} &\lesssim_\varepsilon \la^{\frac{n-1}{2} + \varepsilon}\delta_{K}^{\ell(\frac{n-1}{2} - \frac{n+1}{q})}\Big(\big( \sum_{\xi \in \la^{-1}\tilde{\tau_{\ell}}}|\widehat{f^{\la}}(\xi)|^2 \big)^{\frac{1}{2}}\big( \sum_{\xi \in \la^{-1}\tilde{\tau_{\ell}}'}|\widehat{f^{\la}}(\xi)|^2 \big)^{\frac{1}{2}}\Big)^{\frac{1}{2}} \\
    &= \la^{\frac{n-1}{2} + \varepsilon}\delta_{K}^{\ell(\frac{n-1}{2} - \frac{n+1}{q})}\left( \altlpnorm{f_{\tau_{\ell}}}{2}  \lVert f_{\tau_{\ell}'} \rVert_{2}\right)^{\frac{1}{2}}
\end{split}
\end{equation}
for any ball $B_{\la}$ of radius $\sim \la$.
We will use this to prove Lemma 5.2. Recall that we have $\tau_{\ell} , \tau_{\ell}' \subset \tau_{\ell - 1}$ and $dist(\tau_{\ell}, \tau_{\ell}') \gtrsim \la \dK{\ell}$. Write $[0, \la \delta^{-1}] = \bigsqcup_{i = 1}^{\delta^{-1}}I_i$ where $|I_i| \sim \la$ for each $1 \leq i \leq \delta^{-1}$. Then by using (\ref{extension op vs dilated schrodinger op}), then rescaling before finally using (\ref{estimate for extension op}) on $[0,2 \pi \la]^{n-1} \times I_i$ we have

\begin{equation}
\begin{split}
    \lpnorm{\big|\Sl f_{\tau_{\ell}} \Sl f_{\tau_{\ell}'}\big|^{\frac{1}{2}}}{q}{\Tn \times [0, \delta^{-1}]} &= \Big(\int_{0}^{\delta^{-1}}\int_{[0,2\pi]^{n-1}} \big|\Sl f_{\tau_{\ell}}(x,t)\Sl f_{\tau_{\ell}'}(x,t)\big|^{\frac{q}{2}}dxdt \Big)^{\frac{1}{q}} \\
    &\lesssim \Big( \int_{0}^{\delta^{-1}} \int_{[0,2\pi]^{n-1}} \big| \extensionop{}{f_{\tau_{\ell}}}(\la x, \la t)\extensionop{}{f_{\tau_{\ell}'}}(\la x,\la t)\big|^{\frac{q}{2}}dxdt\Big)^{\frac{1}{q}} \\
    &=  \la^{-\frac{n}{q}}\Big( \int_{0}^{\la \delta^{-1}} \int_{[0,2\pi \la]^{n-1}} \big| \extensionop{}{f_{\tau_{\ell}}}(x, t)\extensionop{}{f_{\tau_{\ell}'}}(x,t)\big|^{\frac{q}{2}}dxdt\Big)^{\frac{1}{q}} \\
    &= \la^{-\frac{n}{q}}\Big(\sum_{i = 1}^{\delta^{-1}} \int_{I_i} \int_{[0,2\pi\la]^{n-1}} \big| \extensionop{}{f_{\tau_{\ell}}}(x, t)\extensionop{}{f_{\tau_{\ell}'}}(x,t)\big|^{\frac{q}{2}}dxdt\Big)^{\frac{1}{q}} \\
    &= \Big( \sum_{i = 1}^{\delta^{-1}}\lpnorm{ \big|\extensionop{}{f_{\tau_{\ell}}} \extensionop{}{f_{\tau_{\ell}'}}\big|^{\frac{1}{2}}}{q}{[0,2\pi \la]^{n-1} \times I_i}^q\Big)^{\frac{1}{q}} \\
    &\lesssim_{\varepsilon} \la^{\frac{n-1}{2} - \frac{n}{q} + \varepsilon}\dK{\ell(\frac{n-1}{2} - \frac{n+1}{q})}\delta^{-\frac{1}{q}}\Big( \altlpnorm{f_{\tau_{\ell}}}{2} \lVert f_{\tau_{\ell}'}\rVert_{2}\Big)^{\frac{1}{2}}
\end{split}
\end{equation}
Thus, the proof of Lemma 5.2 is complete.
\qed

\section{Kernel Estimates}
As a warm-up, we will start by proving kernel estimates for the time-dilated Schr\"{o}dinger operator defined in (\ref{time-dilated schordinger operator}). Recall that $\beta(\xi) = \prod_{i = 1}^{d}\phi(\xi_i)$ where $\phi$ is an even function, $\phi \in C_{0}^{\infty}((-2,2))$ and $\phi \equiv 1$ on $[-1,1]$.
\begin{proposition}
    Let $\kappa > 0$ and suppose $\delta \geq \la^{-1 + \kappa}$. For any $d \geq 1$ we have
    \begin{equation}
        \Sl\big(\Sl\big)^*(x,t;y,s) = \eta(\delta t)\eta(\delta s)\big(\beta^2(D_x/\la)e^{-i\la(t-s)\laplac{\T{d}}}\big)(x,y).
    \end{equation}
    Then for $\la \gg 1$ we have
    \begin{equation}
        \big| \Sl\big(\Sl\big)^*(x,t;y,s) \big| \lesssim\la^{\frac{d}{2}}|t-s|^{-\frac{d}{2}}(1 + |t-s|)^{d}.
    \end{equation}
\end{proposition}

\begin{proof}
    
The case for $d > 1$ follows from the $d = 1$ case. To see this just note that if $x = (x_1,...,x_d)$, $y = (y_1,...,y_d)$ and $k = (k_1,...,k_d)$ then
\begin{equation}
\begin{split}
     \Sl\big(\Sl\big)^*(x,t;y,s) &= \eta(\delta t)\eta(\delta s)\big(\beta^2(D_x/\la)e^{-i\la(t-s)\laplac{\T{d}}}\big)(x,y) \\
     &= \eta(\delta t)\eta(\delta s) \sum_{k \in \Z{d}}\beta^2(k/\la)e^{i\la^{-1}(t-s)|k|^2}e((x-y) \cdot k) \\
     &= \eta(\delta t)\eta(\delta s)\sum_{\substack{k \in \Z{d}; \\ |k_i| \leq 2\la}} \prod_{i = 1}^d\phi^2(k_i / \la)e^{i\la^{-1}(t-s)k_{i}^2}e((x_i - y_i)k_{i}) \\
     &= \eta(\delta t)\eta(\delta s)\prod_{i = 1}^{d}\sum_{\substack{k_i \in \Z{}; \\ |k_i| \leq 2 \la}}\phi^2(k_i / \la)e^{i\la^{-1}(t-s)k_{i}^2}e((x_i - y_i)k_i) \\
     &= \eta(\delta t)\eta(\delta s)\prod_{i = 1}^{d}\big(\phi^2(D_{x_i}/ \la )e^{-i\la^{-1}(t-s)\laplac{\T{}}}\big)(x_i,y_i).
\end{split}
\end{equation}

Thus, it suffices to show that if $x,y \in \T{}$ then 
\begin{equation}
    \big|\eta(\delta t) \eta(\delta s)\big(\phi^2(D_{x} / \la)e^{-i\la^{-1}(t-s)\laplac{\T{}}}\big)(x,y)\big| \lesssim \la^{\frac{1}{2}}|t-s|^{-\frac{1}{2}}(1 + |t-s|)
\end{equation}
We start by setting $\varphi_\la(r; t-s) = \phi^2(r/\la) e^{i\la^{-1}(t-s)r^2}$. Since $\varphi_{\la}$ is an even function of $r$ this implies that 
\begin{equation}
    \phi^2(D_x/\la)e^{-i\la^{-1}(t-s)\laplac{\T{}}} = \frac{1}{2\pi}\int_{-\infty}^{\infty}\widehat{\varphi_{\la}}(\tau; t-s)\cos{\tau\sqrt{-\laplac{\T{}}}}d\tau.
\end{equation}
where 
\begin{equation}
    \widehat{\varphi_{\la}}(\tau; t-s) = \int_{-\infty}^{\infty} e(-r\tau)\varphi_{\la}(r;t-s)dr.
\end{equation}
A simple integration by parts argument shows that 
\begin{equation}
    \partial^{k}_{\tau} \widehat{\varphi_{\la}}(\tau, t-s) = O(\la^{-N}(1 + |\tau|)^{-N}) \ \forall \ N \geq 0 \ \text{if} \ |t-s| \leq 1 \ \text{and} \ |\tau| \geq M
\end{equation}
where $M$ is a sufficiently large constant. A similar argument shows that 
\begin{equation}
\begin{split}
    \partial^{k}_{\tau} \widehat{\varphi_{\la}}(\tau, t-s) = O(\la^{-N}(1 + |\tau|)^{-N}) &\ \forall \ N \geq 0, \\
    &\text{if } \ 2^{j-1} \leq |t-s| \leq 2^j \ \text{and} \ |\tau| \geq M2^j,\ j = 1,2,. \ . \ . \ ,
\end{split}
\end{equation}
if $M$ is fixed large enough. The support properties of the function $\phi$ ensure that it suffices to take $M = 100$. \par
Fix an even function $a \in C_{0}^{\infty}(\mathbb{R})$ such that
\begin{equation}
    a(\tau) = 1 \text{ if } |\tau| \leq 100 \ \text{ and } \ a(\tau) = 0 \ \text{ if } \ |\tau| \geq 200.
\end{equation}
So if we set
\begin{equation}
    \tilde{S}^{\la}_{0}(x,t;y,s) = \frac{1}{2\pi}\int_{-\infty}^{\infty}a(\tau) \widehat{\varphi_{\la}}(\tau, t-s)\big(\cos\tau\sqrt{-\laplac{\T{}}}\big)(x,y)d\tau
\end{equation}
then
\begin{equation}
    \tilde{S}^{\la}_{0}(x,t;y,s) - \big(\phi^2(D_x/\la)e^{-i\la^{-1}(t-s)\laplac{\T{}}}\big)(x,y) = O(\la^{-N}) \  \forall \  N \geq 0 \ \text{ if } \ |t-s| \leq 1.
\end{equation}
Also, if 
\begin{equation}
    \tilde{S}^{\la}_{j}(x,t;y,s) = \frac{1}{2\pi}\int_{-\infty}^{\infty}a(2^{-j}\tau) \widehat{\varphi_{\la}}(\tau, t-s)\big(\cos\tau\sqrt{-\laplac{\T{}}}\big)(x,y)d\tau
\end{equation}
then 
\begin{equation}
\begin{split}
    \tilde{S}^{\la}_{j}(x,t;y,s) - \big(\phi^2(D_x/\la)e^{-i\la^{-1}(t-s)\laplac{\T{}}}\big)(x,y) &= O(\la^{-N}) \  \forall \  N \geq 0 \\
    &\text{if } \ 2^{j-1} \leq |t-s| \leq 2^j, \ j = 1,2, ....
\end{split}
\end{equation}
Thus, we would be done if we should show that 
\begin{equation} \label{estimate for 0 term}
    |\tilde{S}^{\la}_{0}(x,t;y,s)| \lesssim \la^{\frac{1}{2}}|t-s|^{-\frac{1}{2}} \text{ if } |t-s| \leq 1
\end{equation}
and 
\begin{equation} \label{estimate for j term}
    |\tilde{S}^{\la}_{j}(x,t;y,s)| \lesssim \la^{\frac{1}{2}}|t-s|^{-\frac{1}{2}}2^{j} \ \text{ if } \ |t-s| \in [2^{j-1}, 2^j].
\end{equation}
\par
Recall that the Poisson summation formula implies
\begin{equation} \label{Poisson summation for cos}
    \big(\cos\tau\sqrt{-\laplac{\T{}}}\big)(x,y)= \sum_{m \in \Z{}}\big(\cos\tau\sqrt{-\laplac{\R{}}}\big)(x-(y + m)).
\end{equation}
Thus, if we set 
\begin{equation}
\begin{split}
    K^{\la}_{0}(x,t;y,s) &= \frac{1}{2\pi}\int_{-\infty}^{\infty}a(\tau) \widehat{\varphi_{\la}}(\tau, t-s)\big(\cos\tau\sqrt{-\laplac{\R{}}}\big)(x,y)d\tau \\
    &= \frac{1}{(2\pi)^{2}}\int_{-\infty}^{\infty}\int_{-\infty}^{\infty} a(\tau) \widehat{\varphi_{\la}}(\tau, t-s)\cos(\tau|\xi|)e((x-y)\xi)d\tau d\xi
    \end{split}
\end{equation}
then it follows from (\ref{Poisson summation for cos}) that
\begin{equation} \label{Poisson summation formula for 0 term}
    \tilde{S}^{\la}_{0}(x,t;y,s) = \sum_{m \in \Z{}}K^{\la}_{0}(x,t;y+m,s).
\end{equation}
Also, 
\begin{multline}
    K^{\la}_{0}(x,t;y,s) = \frac{1}{(2\pi)^2}\iint a(\tau) \widehat{\varphi_{\la}}(\tau, t-s)\cos(\tau|\xi|)e( (x-y) \xi)d\tau d\xi \\
    = \frac{1}{2\pi}\int \phi^2(\xi/\la)e(\la^{-1}(t-s)|\xi|^2 + (x-y)\xi)d \xi + O(\la^{-N}).
\end{multline}
A standard stationary phase argument (see Chapter 1 of \cite{SoggeFourierIntegrals}) implies that 
\begin{equation} \label{stationary phase estimate for standard kernel}
    \int \phi^2(\xi/\la)e(\la^{-1}(t-s)|\xi|^2 +  (x-y) \xi)d \xi = O\big(\la^{\frac{1}{2}}|t-s|^{-\frac{1}{2}}\big).
\end{equation}
Also, Huygen's Principle and the support properties of $a$ imply that there is some constant $C_1$ such that $K^{\la}_{0}(x,t;y,s) = 0$ if $|x-y| \geq C_1$. Thus, the number of nonzero summands in (\ref{Poisson summation formula for 0 term}) is $O(1)$. Combining these two facts together yields (\ref{estimate for 0 term}).
\par
If we set 
\begin{equation}
    K^{\la}_{j}(x,t;y,s) = \frac{1}{2\pi}\int_{-\infty}^{\infty}a(2^{-j}\tau) \widehat{\varphi_{\la}}(\tau, t-s)\cos\tau\sqrt{-\laplac{\R{}}}(x,y)d\tau
\end{equation}
then it follows from (\ref{Poisson summation for cos}) that 
\begin{equation} \label{Poisson summation formula for j term}
    \tilde{S}^{\la}_{j}(x,t;y,s) = \sum_{m \in \Z{}}K^{\la}_{j}(x,t;y+m,s).
\end{equation}
When $|t-s| \in [2^{j-1}, 2^j]$ we also have
\begin{multline}
    K^{\la}_{j}(x,t;y,s) = \frac{1}{(2\pi)^2}\iint a(2^{-j}\tau) \widehat{\varphi_{\la}}(\tau, t-s)\cos(\tau|\xi|)e(( x-y)\xi)d\tau d\xi \\
    = \frac{1}{2\pi}\int \phi^2(\xi/\la)e(\la^{-1}(t-s)|\xi|^2 + ( x-y)\xi)d \xi + O(\la^{-N}).
\end{multline}
Huygen's principle and the support properties of $a$ imply that $K^{\la}_{j}(x,t;y,s) =0$ if $|x - y| \geq 2^j C_1$ so the number of nonzero summands in (\ref{Poisson summation formula for j term}) is $O(2^{j})$. Combining this with (\ref{stationary phase estimate for standard kernel}) yields (\ref{estimate for j term}).
\end{proof}

We now turn to proving the kernel estimates for our time-dilated Schr\"{o}dinger operators with additional frequency localization. Fix $m \in \mathbb{Z}$ with $1 \leq m \leq K$. Let $I = [c - \la \dK{m}, c + \la \dK{m}] \subset [-\la, \la]$ and let $\chi \in C_{0}^{\infty}((-2,2))$ such that $\chi \equiv 1$ on $[-1,1]$. Set $\chi_{I}(\xi) =\chi((\la \dK{m})^{-1}(\xi - c))$. Then $\supp\chi_{I} \subset 2I$ and $\chi_{I} \equiv 1$ on $I$. Define the following Fourier multiplier on $\T{}$:
\begin{equation}
    \chi_{I}(D_x)g(x) = \sum_{j \in \Z{}}\widehat{g}(j)\chi_{I}(j)e(x j), \ x \in \T{}.
\end{equation}
To prove the kernel estimates we need, it suffices to prove the following:

\begin{proposition}
Let $\kappa > 0$ and suppose $\delta \geq \la^{-1 + \kappa}$. Then for $x,y \in \T{}$ we have
\begin{equation} 
     \big|\eta(\delta t)\eta(\delta s)\big( \chi_{I}^{2}(D_x)\phi^2(D_x / \la)e^{-i\la^{-1}(t-s)\laplac{\T{}}} \big)(x,y)\big| \lesssim \la^{\frac{1}{2}}|t-s|^{-\frac{1}{2}}\big(1 + \dK{m}|t-s|\big).
\end{equation} 
Thus, when $m = K$ (so that $\dK{m} = \delta$), if $|t-s| \leq 2\delta^{-1}$ we have 
\begin{equation}
    \big|\eta(\delta t)\eta(\delta s)\big( \chi_{I}^{2}(D)\phi^2(D / \la)e^{-i\la^{-1}(t-s)\laplac{\T{}}} \big)(x,y) \big| \lesssim \la^{\frac{1}{2}}|t-s|^{-\frac{1}{2}}.
\end{equation}
\end{proposition}

\begin{proof}
First, fix $x,y \in \T{}$ and note that $\partial_{\xi}^{k}\chi_{I}(\xi) = O\big( (\la \dK{m})^{-k}\big)$. Also, this function is supported in an interval of length $\sim \la \dK{m}$ so  $\check{\chi}_{I}(x)$ =  $O\big( (\la \dK{m})(1+\la \dK{m}|x|)^{-N}\big)$ for any $N \geq 0$. It then follows from the Poisson Summation formula that the kernel of our operator satisfies
$\chi_I(D_x)(x,y) = O\Big(\la \dK{m}\big(1 + \la \dK{m}|x- y|\big)^{-N}\Big)$ for any $N \geq 0$. It follows that
\begin{equation} \label{kernel bounds for angular cutoff multipliers}
    \int_{\mathbb{R}}|\check{\chi}_{I}(z)|dz, \ \sup_{y \in \mathbb{T}}\int_{\mathbb{T}} |\chi_{I}(D_x)(x,y)|dx , \ \sup_{x \in \mathbb{T}}\int_{\mathbb{T}} |\chi_{I}(D_x)(x,y)|dy \leq C.
\end{equation}
Thus, it follows from Proposition 6.1 that we obtain the desired bounds in Proposition 6.2 if $|t-s| \leq 1$.
\par
Also, (\ref{kernel bounds for angular cutoff multipliers}) implies that if we let 
\begin{equation}
    K^{\la}_{I}(x,t;y,s) = \eta(\delta t)\eta(\delta s)\big(\chi_{I}(D_x)\phi^{2}(D_x / \la)e^{-i\la^{-1}(t-s)\laplac{\T{}}}\big)(x,y)
\end{equation}
then it suffices to see that this kernel satisfies the desired bounds if $|t-s| \geq 1$. Also, recall that
\begin{equation} 
    \chi_I(\xi) = \chi\big((\la \dK{m})^{-1}(\xi - c)\big)
\end{equation}
and
\begin{equation} \label{support property of chi_I}
    \supp{\chi_I} \subset 2I = [c -2\la\dK{m}, c+2\la \dK{m} ]
\end{equation}
\par
If we let $K^{\la}(x,t;y,s)$ be the kernel of the multiplier operator $\eta(\delta t)\eta(\delta s)\phi^2(D_x/\la)e^{-i\la^{-1}(t-s)\laplac{\T{}}}$ then using the Poisson Summation formula and writing $\chi_{I}$ in terms of its inverse Fourier transform we have
\begin{equation} \label{K^(la, theta) Poisson summation 1}
\begin{split}
    K^{\la}_{I} (x,t;y,s) &= \eta(\delta t) \eta(\delta s)\frac{1}{2 \pi} \sum_{j \in \Z{}}\int_{\mathbb{R}} \chi_{I}(\xi)\phi^2(\xi/\la)e((x-y-j)\xi + \la^{-1}(t-s)|\xi|^2)d\xi\\
    &= \eta(\delta t) \eta(\delta s) \frac{1}{2\pi}\int_{\mathbb{R}} \check{\chi}_{I} (x - z)K^{\la}(z,t; y, s) dz.
\end{split}
\end{equation}
Further, if we repeat the arguments from the proof of Proposition 6.1 then we see that for $|t-s| \in [2^{j-1}, 2^j]$ we have
\begin{equation} \label{K^la poisson summation}
    K^{\la}(z,t; y, s) = \sum_{j \in \Z{}}\tilde{K}^{\la}(z,t;y+j,s) + O(\la^{-N})
\end{equation}
where
\begin{equation} \label{tilde K^la definition} 
    \tilde{K}^{\la}(z,t; y, t) = \frac{1}{2\pi} \int_{\mathbb{R}} \phi^2(\xi/\la)e(\la^{-1}(t-s)|\xi|^2 + (z-y) \xi)d\xi.
\end{equation}
As before, the sum in (\ref{K^la poisson summation}) is over the $j \in \Z{}$ such that $|z - (y + j)| \lesssim |t-s|$ due to Huygen's principle (in particular, the sum is finite). It follows that 
\begin{equation} \label{K^(la, theta) Poisson summation 2}
    K^{\la}_{I} (x,t;y,s) = \eta(\delta t) \eta(\delta s) \sum_{j \in \Z{}}\int_{\mathbb{R}} \check{\chi_{I}} \ (x - z)\tilde{K}^{\la}(z,t; y+ j, s) dz + O\big(\la^{-N}\big)
\end{equation}
We recall the estimate (\ref{stationary phase estimate for standard kernel}), which says that 
\begin{equation} \label{stationary phase bound for K^la}
    \big| \tilde{K}^{\la}(z,t;y,s)\big| \lesssim \la^{\frac{1}{2}}|t-s|^{-\frac{1}{2}}.
\end{equation}
It follows from (\ref{kernel bounds for angular cutoff multipliers}) and (\ref{stationary phase bound for K^la}) that for each fixed $j$ we have
\begin{equation}
    \Big|\int_{\mathbb{R}} \check{\chi}_{I} \ (x - z)\tilde{K}^{\la}(z,t; y+ j, s) dz\Big| \lesssim \la^{\frac{1}{2}}|t-s|^{-\frac{1}{2}}
\end{equation}
Thus, in order to finish the proof of Proposition 6.2 it suffices to see that that number of non-negligible summands in (\ref{K^(la, theta) Poisson summation 2}) is $O\big(\dK{m}|t-s|\big)$.
\par
Note that 
\begin{equation}
    \int_{\mathbb{R}} \check{\chi}_{I} \ (x - z)\tilde{K}^{\la}(z,t; y+ j, s) dz = \int_{\mathbb{R}} \chi_{I}(\xi)\phi^2(\xi/\la)e(( x-y-j) \xi + \la^{-1}(t-s)|\xi|^2)d\xi.
\end{equation}
Recalling (\ref{support property of chi_I}) we can integrate by parts to see that this integral is $O(\la^{-N})$ for all $N \geq 0$ if $j$ is not contained in the set
\begin{equation}
    \{ j \in \Z{} \ : \ |x- y - j|\in \big[2|t-s|(\la^{-1}c - C_0\dK{m}), 2|t-s|(\la^{-1}c + C_0\dK{m})\big] \}
\end{equation}  
for fixed $x,y,t$ and $s$, where $C_0$ is a large fixed constant. This is an interval of length $\sim \dK{m}|t-s|$ so we have that
\begin{equation}
    \#\{ j \in \Z{} \ : \ |x- y - j|\in \big[2|t-s|(\la^{-1}c - C\dK{m}), 2|t-s|(\la^{-1}c + C\dK{m})\big] \} = O\big(\dK{m}|t-s|\big).
\end{equation} Thus, the proof of Proposition 6.2 is complete.
\end{proof}

Proposition 6.2 immediately implies the analogous $d$-dimensional estimates that we require. Let $\tau = \prod_{i = 1}^d I_i$ be a cube of side-length $2\la \dK{m}$ centered at $c = (c_1,...,c_d)$ contained in $[-\la, \la]^{d}$ with $I_i = [c_i - \la \dK{m}, c_i + \la \dK{m}]\subset [-\la, \la]$. For $\xi = (\xi_1,...,\xi_d)$ let
\begin{equation}
    \chi_{\tau}(\xi) = \prod_{i=1}^{d}\chi_{I_i}(\xi_i)
\end{equation}
Note that 
\begin{equation}
    \chi_{\tau} \equiv 1 \ \text{on} \ \tau \ \text{and} \ \text{supp}\chi_{\tau} \subset 2\tau.
\end{equation}
As in the one-dimensional case, define the Fourier multiplier on $\T{d}$
\begin{equation}
    \chi_{\tau}(D_x)g(x) = \sum_{k \in \Z{d}}\widehat{g}(k)\chi_{\tau}(k)e(x \cdot k), \ x \in \T{d}.
\end{equation}
Then, using a similar argument to the one used to reduce the $d$-dimensional estimate in Proposition 6.1 to its $1$-dimensional analog, Proposition 6.2 implies the following:
\begin{corollary}
    Let $\kappa > 0$ and suppose $\delta \geq \la^{-1 + \kappa}$. If $x,y \in \T{d}$ then
    \begin{equation}
        \big|\eta(\delta t)\eta(\delta s)\big( \chi_{\tau}^{2}(D_x)\beta^2(D_x / \la)e^{-i\la^{-1}(t-s)\laplac{\T{d}}} \big)(x,y)\big| \lesssim \la^{\frac{d}{2}}|t-s|^{-\frac{d}{2}}\big(1 + \dK{m}|t-s|\big)^d.
    \end{equation}

Thus, when $m = K$ (so that $\dK{m} = \delta$), since $|t-s| \leq 2\delta^{-1}$ we have 
\begin{equation}
    \big|\eta(\delta t)\eta(\delta s)\big( \chi_{\tau}^{2}(D_x)\beta^2(D_x / \la)e^{-i\la^{-1}(t-s)\laplac{\T{d}}} \big)(x,y) \big| \lesssim \la^{\frac{d}{2}}|t-s|^{-\frac{d}{2}}.
\end{equation}
\end{corollary}

\section{Proof of Theorem 1.1}
We now turn to the Proof of Theorem 1.1. As mentioned earlier, the proof follows from applying Proposition 3.1 for $1 \leq \ell < K$. We also need an estimate when $\ell = K$. Note that when $\ell = K$ we have $\dK{K} = \delta$. Let $\tau \subset Q_{\la}$ be a cube of side-length $2\la \delta$. Let $\chi_{\tau}$ be defined as in Section 6. Recall that 
\begin{equation}
    \Sl \chi_{\tau}(D_x)f(x, t) = S_{\la, t} \chi_{\tau}(D_x)f(x). 
\end{equation}
Further, since
\begin{equation}
    S_{\la, t} \chi_{\tau}(D_x) \big(S_{\la, s} \chi_{\tau}(D_x)\big)^{*}(x,y) = \eta(\delta t)\eta(\delta s)\big(\chi_{\tau}^2(D_x)\beta^{2}(D_x / \la)e^{-i \la^{-1}(t-s)\laplac{\Tn}}\big)(x,y)
\end{equation}
we can use the kernel estimates in Corollary 6.0.1 to conclude that
\begin{equation}
    \lpnorm{S_{\la, t}\chi_{\tau}(D_x) \big( S_{\la,s} \chi_{\tau}(D_x)\big)^{*} f}{\infty}{\Tn} \lesssim\la^{\frac{n-1}{2}}|t-s|^{-\frac{n-1}{2}}\lpnorm{f}{1}{\Tn}
\end{equation}
and by Parseval's Theorem
\begin{equation}
    \lpnorm{S_{\la, t}\chi_{\tau}(D_x)f}{2}{\Tn} \lesssim \lpnorm{f}{2}{\Tn}.
\end{equation}
Then we can use a $TT^{*}$ argument such as in Corollary 0.3.7 in \cite{SoggeFourierIntegrals} to conclude that
\begin{equation} \label{K scale estimate}
     \lpnorm{\Sl \chi_{\tau}(D_x)f}{q_c}{\Tn \times [0, \delta^{-1}]} \lesssim \la^{\frac{1}{q_c}}\altlpnorm{f}{2}.
\end{equation}

Now, suppose $\tau_{K-1} \subset Q_{\la}$ is a cube of side-length $2\la \dK{K-1}$. Then by Proposition 3.1 we have
\begin{equation} \label{broad-narrow at scale K}
\begin{split}
    \lpnorm{\Sl f_{\tau_{K-1}}}{q_c}{\Tn\times[0, \delta^{-1}]} \leq \big(C_1 \la^{\frac{1}{q_c}} + C_{\varepsilon}\la^{\frac{n-1}{n}\frac{1}{q_c} + \varepsilon}\dK{-C}\delta^{-\frac{1}{q_c}}\big)\altlpnorm{f_{\tau_{k-1}}}{2} \\
    + C_2\Big(\sum_{\tau_{K} \subset \tau_{K-1}}\lpnorm{\Sl f_{\tau_{K}}}{q_c}{\Tn \times [0, \delta^{-1}]}^{2}\Big)^{\frac{1}{2}}
\end{split}
\end{equation}
\noindent
By (\ref{K scale estimate}) we have
\begin{equation}\label{smallest scale Strichartz}
    \lpnorm{\Sl f_{\tau_{K}}}{q_c}{\Tn \times [0, \delta^{-1}]} = \lpnorm{\Sl \chi_{\tau_{K}}(D_x)f_{\tau_{K}}}{q_c}{\Tn \times [0, \delta^{-1}]}   \lesssim \la^{\frac{1}{q_c}}\altlpnorm{f_{\tau_{K}}}{2}
\end{equation}
Combining this with (\ref{broad-narrow at scale K}) and using the orthogonality of the functions $\{ f_{\tau_{K}}\}_{\tau_K \subset \tau_{K-1}}$ we have
\begin{equation}
    \lpnorm{\Sl f_{\tau_{K-1}}}{q_c}{\Tn\times[0, \delta^{-1}]} \lesssim_{\varepsilon}\big( \la^{\frac{1}{q_c}} + \la^{\frac{n-1}{n}\frac{1}{q_c} + \varepsilon}\delta^{-\frac{n+1}{n}\frac{1}{q_c}}\big)\altlpnorm{f_{\tau_{K-1}}}{2}.
\end{equation}

Now, suppose we know that 
\begin{equation} \label{scale l estimate}
    \lpnorm{\Sl f_{\tau_{\ell}}}{q_c}{\Tn\times[0, \delta^{-1}]} \leq C_{\varepsilon}\big( \la^{\frac{1}{q_c}} + \la^{\frac{n-1}{n}\frac{1}{q_c} + \varepsilon}\dK{-C}\delta^{-\frac{n+1}{n}\frac{1}{q_c}}\big)\altlpnorm{f_{\tau_{\ell}}}{2}
\end{equation}
for some $\ell$ with $1 \leq \ell \leq K$ and any cube $\tau_{\ell} \subset Q_{\la}$ of side-length $2\la \dK{\ell}$. Fix a cube $\tau_{\ell - 1} \subset Q_{\la}$. Using Proposition 3.1, (\ref{scale l estimate}) and the orthogonality of the functions $\{f_{\tau_{\ell}}\}_{\tau_{\ell}\subset \tau_{\ell - 1}}$ we have
\begin{equation}
\begin{split}
    \lpnorm{\Sl f_{\tau_{\ell-1}}}{q_c}{\Tn\times[0, \delta^{-1}]} \leq \big(C_1 \la^{\frac{1}{q_c}} + C_{\varepsilon}\la^{\frac{n-1}{n}\frac{1}{q_c} + \varepsilon}\dK{-C}\delta^{-\frac{n+1}{n}\frac{1}{q_c}}\big)\altlpnorm{f_{\tau_{\ell-1}}}{2} \\
    + C_2\Big(\sum_{\tau_{\ell} \subset \tau_{\ell-1}}\lpnorm{\Sl f_{\tau_{\ell}}}{q_c}{\Tn \times [0, \delta^{-1}]}^{2}\Big)^{\frac{1}{2}} \\
    \lesssim_{\varepsilon}\big( \la^{\frac{1}{q_c}} +\la^{\frac{n-1}{n}\frac{1}{q_c} + \varepsilon}\dK{-C}\delta^{-\frac{n+1}{n}\frac{1}{q_c}}\big)\altlpnorm{f_{\tau_{\ell-1}}}{2}
\end{split}
\end{equation}
Iterating this argument (starting with $\ell = K$ and working back up until $\ell = 1$ so that $\tau_{\ell - 1} = \tau_{0} = Q_{\la}$) yields that for any function $ f(x) = \sum_{k \in \tilde{Q_{\la}}}\widehat{f}(k)e(x \cdot k)$ we have
\begin{equation}
    \lpnorm{\Sl f}{q_c}{\Tn \times [0, \delta^{-1}]} \leq C_{\varepsilon, K}\big( \la^{\frac{1}{q_c}} +\la^{\frac{n-1}{n}\frac{1}{q_c} + \varepsilon}\dK{-C}\delta^{-\frac{n+1}{n}\frac{1}{q_c}}\big)\altlpnorm{f}{2}.
\end{equation}
Note that the constant $C_{\varepsilon, K}$ depends on $\varepsilon$ and $K$ (in particular, the constant depends on the  number of times we iterate the above argument, which is exactly $K$). Now, choose $K$ large enough so that $\frac{C}{K} < \varepsilon$. Then $\dK{-C} = \delta^{-\frac{C}{K}} < \delta^{-\varepsilon} < \la^{\varepsilon}$. Further, note that 
\begin{equation}
    \la^{\frac{n-1}{n}\frac{1}{q_c} + \varepsilon}\dK{-C}\delta^{-\frac{n+1}{n}\frac{1}{q_c}} \leq \la^{\frac{n-1}{n}\frac{1}{q_c} + 2\varepsilon}\delta^{-\frac{n+1}{n}\frac{1}{q_c}} \leq \la^{\frac{1}{q_c}}, \ \text{if} \ \delta \geq \la^{-\frac{1}{n+1} + \frac{4n}{n-1}\varepsilon}.  
\end{equation}

Thus, for any $\varepsilon > 0$, if $\delta \geq \la^{-\frac{1}{n+1} + \varepsilon}$ then 
\begin{equation}
    \lpnorm{\Sl f}{q_c}{\Tn \times [0, \delta^{-1}]} \lesssim_{\varepsilon} \la^{\frac{1}{q_c}}\altlpnorm{f}{2}
\end{equation}
and by rescaling this implies
\begin{equation}
    \lpnorm{e^{-it \laplac{\Tn}}\beta(D_x/\la)f}{q_c}{\Tn \times [0, \frac{1}{\la \delta}]} \lesssim_{\varepsilon} \altlpnorm{f}{2}
\end{equation}
so the proof of Theorem 1.1 on the square torus is complete.
\qed

\section{Endpoint Estimate on the Square Torus}
We now show how to adjust the argument given in Sections 3-7 to prove Theorem 1.2 on the square torus $\Tn$. As in the proof of Theorem 1.1, by rescaling, it suffices to show that for any $\varepsilon > 0$ we have
\begin{equation}
    \mixlpnorm{\Sl f}{2}{q_e}{\Tn\times[0,\delta^{-1}]} \lesssim_{\varepsilon} \la^{\frac{1}{2}}\lpnorm{f}{2}{\Tn}
\end{equation}
if $\delta \geq \la^{-\frac{n-3}{(n-1)(n+3)}+ \varepsilon}$.
Recall that $q_e = \frac{2d}{d-2} = \frac{2(n-1)}{n-3}$. The proof of Theorem 1.2 follows the same scheme as the proof of Theorem 1.1. We use a multi-scale analysis, bilinear restriction estimates and height-splitting at each scale with $1 \leq \ell < K$. We start by proving the analogs of Propositions 3.2 and 3.3. Fix a cube $\tau_{\ell - 1} \subset Q_{\la}$ of sidelength $\la\dK{\ell}$. As before, we may assume that $\altlpnorm{f_{\tau_{\ell - 1}}}{2} = 1$.

\subsection{Large Height Estimate at Scale $\dK{\ell-1}$}
We start by proving the analog of Proposition 3.2:
\begin{equation} \label{large height endpoint estimate}
    \mixlpnorm{\Sl f_{\tau_{\ell - 1}}}{2}{q_e}{A^{+}_{\tau_{\ell - 1}}} \lesssim \la^{\frac{1}{2}}\altlpnorm{f_{\tau_{\ell - 1}}}{2}.
\end{equation}
First, let $\chi_{\tau_{\ell - 1}}$ be the function described in the first paragraph of Section 4 and let $\chi_{\tau_{\ell - 1}}(D_x)$ denote the associated Fourier multiplier. As a  reminder, we have $\Sl f_{\tau_{\ell - 1}} = \Sl \chi_{\tau_{\ell - 1}}(D_x) f_{\tau_{\ell - 1}}$. Now, choose a function $g$ such that
\begin{equation} \label{endpoint large height duality function}
\begin{split}
        \mixlpnorm{g}{2}{q_e'}{\Tn \times [0, \delta^{-1}]} &= 1 \ \text{and} \\
        \mixlpnorm{\Sl \chi_{\tau_{\ell - 1}}(D_x)f_{\tau_{\ell - 1}}}{2}{q_e}{A^{+}_{\tau_{\ell - 1}}} &= \Big|\iint \Sl \chi_{\tau_{\ell - 1}}(D_x) f_{\tau_{\ell - 1}}(x,t)\mathbbm{1}_{A^{+}_{\tau_{\ell - 1}}}(x,t)\overline{ g(x,t)} dx dt\Big|.
\end{split}
\end{equation}
Since $\altlpnorm{f_{\tau_{\ell - 1}}}{2} = 1$, then by repeating some of our calculations from Section 4 we have
\begin{equation}
\begin{split}
\mixlpnorm{\Sl  f_{\tau_{\ell - 1}}}{2}{q_e}{A^{+}_{\tau_{\ell - 1}}}^2 &= \mixlpnorm{\Sl \chi_{\tau_{\ell - 1}}(D_x) f_{\tau_{\ell - 1}}}{2}{q_e}{A^{+}_{\tau_{\ell - 1}}}^2 \\
    &= \Big|\int f_{\tau_{\ell - 1}}(x)\overline{\big(\Sl \chi_{\tau_{\ell - 1}}(D_x)\big)^* (\mathbbm{1}_{A^{+}_{\tau_{\ell - 1}}}g)(x)} dx \Big|^2 \\
     &\leq \iint \Sl \chi_{\tau_{\ell -1}}(D_x) \big( \Sl \chi_{\tau_{\ell - 1}}(D_x)\big)^*\big( \mathbbm{1}_{A^{+}_{\tau_{\ell - 1}}}g\big)(x,t) \overline{\big( \mathbbm{1}_{A^{+}_{\tau_{\ell - 1}}}g\big)(x,t)}dxdt \\
    &=\iint L_{\tau_{\ell - 1}}^{\la}\big( \mathbbm{1}_{A^{+}_{\tau-{\ell - 1}}}g\big)(x,t) \overline{\big( \mathbbm{1}_{A^{+}_{\tau_{\ell - 1}}}g\big)(x,t)}dxdt \\ 
    &+ \iint G_{\tau_{\ell - 1}}^{\la}\big( \mathbbm{1}_{A^{+}_{\tau_{\ell - 1}}}g\big)(x,t) \overline{\big( \mathbbm{1}_{A^{+}_{\tau_{\ell - 1}}}g\big)(x,t)}dxdt \\ 
    &= I + II
\end{split}
\end{equation}
where $L_{\tau_{\ell -1}}^{\la}$ is defined as in (\ref{large height local kernel I}) and (\ref{large height local kernel II}).
Once again, the classical Strichartz estimate of Burq-Gerard-Tzvetkov \cite{bgtstrichartzmanifold} implies that
\begin{equation}
    \mixlpnorm{L_{\tau_{\ell -1}}^{\la}H}{2}{q_e}{\Tn \times [0,\delta^{-1}]} \lesssim \la \mixlpnorm{H}{2}{q_e'}{\Tn \times [0,\delta^{-1}]}.
\end{equation}
Using this after applying H\"{o}lder's inequality in the $x$-integral and the Cauchy-Schwarz inequality in the $t$-integral, we see that 
\begin{equation} \label{endpoint large height 1st term estimate}
    \begin{split}
        \big| I \big| &\leq \mixlpnorm{L_{\tau_{\ell - 1}}^{\la}\big( \mathbbm{1}_{A^{+}_{\tau_{\ell - 1}}}g\big)}{2}{q_e}{\Tn \times [0,\delta^{-1}]}  \mixlpnorm{\mathbbm{1}_{A^{+}_{\tau_{\ell - 1}}}g}{2}{q_e'}{\Tn \times [0,\delta^{-1}]} \\
        &\lesssim \la \mixlpnorm{\mathbbm{1}_{A^{+}_{\tau_{\ell - 1}}}g}{2}{q_e'}{\Tn \times [0,\delta^{-1}]}^2 \leq \la\mixlpnorm{g}{2}{q_e'}{\Tn \times [0,\delta^{-1}]}^2 = \la.
    \end{split}
\end{equation}
By Corollary 6.0.1, we have
\begin{equation} \label{endpoint kernel estimate for large heights}
\begin{split}
    \big| \Sl \chi_{\tau_{\ell - 1}}(D_x) \big(\Sl \chi_{\tau_{\ell - 1}}(D_x)\big)^*(x,t;y,s) \big|  \lesssim \delta_{K}^{(n-1)(\ell-1)} \big(\la \delta^{-1}\big)^{\frac{n-1}{2}}, \ \text{if } |t-s| \leq 2\delta^{-1}.
\end{split}
\end{equation}
It follows that 
\begin{equation}
    \Lpopnormalt{G_{\tau_{\ell - 1}}^{\la}}{1}{x,t}{\infty}{x,t} \leq C\delta_{K}^{(n-1)(\ell - 1)}(\la \delta^{-1})^{\frac{n-1}{2}}
\end{equation}
for some constant $C$. Using this and then applying H\"{o}lder's inequality in the $x$-integral followed by the Cauchy-Schwartz inequality in the $t$-integral, we see that
\begin{equation}
\begin{split}
    \big| II \big| &\leq C\delta_{K}^{(n-1)\ell}(\la \delta^{-1})^{\frac{n-1}{2}} \altlpnorm{\mathbbm{1}_{A^{+}_{\tau_{\ell - 1}}}g}{1}^2 \\
    &\leq C\delta_{K}^{(n-1)(\ell-1)}(\la \delta^{-1})^{\frac{n-1}{2}}\mixlpnorm{g}{2}{q_e'}{\Tn \times [0,\delta^{-1}]}^{2} \mixlpnorm{\mathbbm{1}_{A^{+}_{\tau_{\ell - 1}}}}{2}{q_e}{\Tn \times [0,\delta^{-1}]}^2 \\
    &= C\delta_{K}^{(n-1)(\ell-1)}(\la \delta^{-1})^{\frac{n-1}{2}}\mixlpnorm{\mathbbm{1}_{A^{+}_{\tau_{\ell - 1}}}}{2}{q_e}{\Tn \times [0,\delta^{-1}]}^2.
\end{split}
\end{equation}
The definition of the set $A^{+}_{\tau_{\ell - 1}}$ in (\ref{height decomposition definitions}) yields
\begin{equation}
\begin{split}
    \mixlpnorm{\mathbbm{1}_{A^{+}_{\tau_{\ell - 1}}}}{2}{q_e}{\Tn \times [0,\delta^{-1}]}^2 &\leq \Big(C_0\big(\la \delta^{-1}\delta_{K}^{2(\ell - 1)}\big)^{\frac{n-1}{4}}\Big)^{-2} \mixlpnorm{\Sl f_{\tau_{\ell - 1}}}{2}{q_e}{A^{+}_{\tau_{\ell - 1}}}^2 \\
    &= \Big(C_0^2 (\la \delta^{-1})^{\frac{n-1}{2}}\delta_{K}^{(n-1)(\ell - 1)}\Big)^{-1}\mixlpnorm{\Sl f_{\tau_{\ell - 1}}}{2}{q_e}{A^{+}_{\tau_{\ell - 1}}}^2.
\end{split}
\end{equation}
Choosing $C_0$ large enough so that $C_0^{-2}C \leq \frac{1}{2}$ we get 
\begin{equation}
    \big| II \big| \leq \frac{1}{2}\mixlpnorm{\Sl f_{\tau_{\ell - 1}}}{2}{q_e}{A^{+}_{\tau_{\ell - 1}}}^2.
\end{equation}
Combining this with our estimate for $I$ in (\ref{endpoint large height 1st term estimate}) we have
\begin{equation}
    \mixlpnorm{\Sl f_{\tau_{\ell - 1}}}{2}{q_e}{A^{+}_{\tau_{\ell - 1}}}^2 \leq C \la \altlpnorm{f_{\tau_{\ell - 1}}}{2}^2 + \frac{1}{2}\mixlpnorm{\Sl f_{\tau_{\ell  -1}}}{2}{q_e}{A^{+}_{\tau_{\ell - 1}}}^2
\end{equation}
which yields (\ref{large height endpoint estimate}).
\subsection{Small Height Estimate at Scale $\dK{\ell-1}$}
We now prove the analog of Proposition 3.2. We start by using our broad-narrow decomposition (\ref{broad-narrow decomposition}) to write
\begin{equation}
\begin{split}
    &\mixlpnorm{\Sl f_{\tau_{\ell - 1}}}{2}{q_e}{A^{-}_{\tau_{\ell - 1}}}^2 = \int\Big(\int |\Sl f_{\tau_{\ell - 1}}(x,t)|^{q_e - q}|\Sl f_{\tau_{\ell - 1}}(x,t)|^{q}\mathbbm{1}_{A_{\tau_{\ell - 1}}^{-}}(x,t)dx\Big)^{\frac{2}{q_e}}dt \\
    &=  \int\Big(\int |\Sl f_{\tau_{\ell - 1}}(x,t)|^{q_e - q}\Big|\sum_{\tau_{\ell} \subset \tau_{\ell - 1}}\Sl f_{\tau_{\ell}}(x,t)\Big|^{q}\mathbbm{1}_{A_{\tau_{\ell - 1}}^{-}}(x,t)dx\Big)^{\frac{2}{q_e}}dt \\ 
    &\leq C_1^{\frac{2q}{q_e}}\int \Big(\int|\Sl f_{\tau_{\ell - 1}}(x,t)|^{q_e - q}\Big(\sum_{\tau_{\ell} \subset \tau_{\ell - 1}}|\Sl f_{\tau_{\ell}}(x,t)|^{2}\Big)^{\frac{q}{2}}\mathbbm{1}_{A_{\tau_{\ell - 1}}^{-}}(x,t)dx\Big)^{\frac{2}{q_e}}dt \\
    &+ C_2^{\frac{2q}{q_e}}\dK{-C\frac{2q}{q_e}} \int\Big( \sum_{\substack{\tau_{\ell}, \tau_{\ell}' \subset \tau_{\ell - 1} \\ \text{dist}(\tau_{\ell} , \tau_{\ell}') \gtrsim \la\dK{\ell}}}\int |\Sl f_{\tau_{\ell - 1}}(x,t)|^{q_e - q}  |\Sl f_{\tau_{\ell}}(x,t)\Sl f_{\tau_{\ell}'}(x,t)|^{\frac{q}{2}}\mathbbm{1}_{A_{\tau_{\ell - 1}}^{-}}(x,t)dx\Big)^{\frac{2}{q_e}}dt \\ 
    & = I + II.
\end{split}
\end{equation}
Just as in Section 5, let $Q_{\ell}f_{\tau_{\ell - 1}}(x,t) = \big( \sum_{\tau_{\ell} \subset \tau_{\ell - 1}}|\Sl f(x,t)|^2\big)^{\frac{1}{2}}$. Then
\begin{equation} \label{narrow term endpoint estimate}
\begin{split}
    I &= C_1^{\frac{2q}{q_e}}\int \Big(\int|\Sl f_{\tau_{\ell - 1}}(x,t)|^{q_e - q}|\Sq{f_{\tau_{\ell - 1}}}{}{\ell}(x,t)|^{q}\mathbbm{1}_{A_{\tau_{\ell - 1}}^{-}}(x,t)dx\Big)^{\frac{2}{q_e}}dt \\
    &\leq C_1^{\frac{2q}{q_e}}\mixlpnorm{\Sl f_{\tau_{\ell - 1}}}{2}{q_e}{A_{\tau_{\ell - 1}}^{-}}^{2(1 - \frac{q}{q_e})} \mixlpnorm{\Sq{f_{\tau_{\ell - 1}}}{}{\ell}}{2}{q_e}{\Tn \times [0, \delta^{-1}]}^{\frac{2q}{q_e}} \\
    &\leq \frac{1}{2}\mixlpnorm{\Sl f_{\tau_{\ell - 1}}}{2}{q_e}{A_{\tau_{\ell - 1}}^{-}}^{2} + C\mixlpnorm{\Sq{f_{\tau_{\ell - 1}}}{}{\ell}}{2}{q_e}{\Tn \times [0, \delta^{-1}]}^{2} \\
    &\leq \frac{1}{2}\mixlpnorm{\Sl f_{\tau_{\ell - 1}}}{2}{q_e}{A_{\tau_{\ell - 1}}^{-}}^{2} + C\sum_{\tau_{\ell} \subset \tau_{\ell - 1}} \mixlpnorm{\Sl f_{\tau_{\ell}}}{2}{q_e}{\Tn \times [0, \delta^{-1}]}^2
\end{split} 
\end{equation}
where we used H\"{o}lder's inequality to go from line 1 to line 2, then Young's inequality (with $\epsilon = \big(2 C_{1}^{\frac{2q}{q_e}}\big)^{-1}$, see Appendix B of \cite{EvansPDE} for more details) in the second-to-last line and, finally, Minkowski's inequality in the last line.
Further, using H\"{o}lder's inequality in the $t$-integral followed by Lemma 5.2 we have
\begin{equation} \label{broad term endpoint estimate}
    \begin{split}
        &II \lesssim \dK{-C\frac{2q}{q_e}}\lpnorm{\Sl f_{\tau_{\ell - 1}}}{\infty}{A_{\tau_{\ell - 1}}^{-}}^{\frac{2}{q_e}(q_e - q)}\int_{0}^{\delta^{-1}}\Big( \sum_{\substack{\tau_{\ell}, \tau_{\ell}' \subset \tau_{\ell - 1} \\ \text{dist}(\tau_{\ell} , \tau_{\ell}') \gtrsim \la\dK{\ell}}}\int_{\T{n-1}}   |\Sl f_{\tau_{\ell}}(x,t)\Sl f_{\tau_{\ell}'}(x,t)|^{\frac{q}{2}}dx\Big)^{\frac{2}{q_e}}dt \\
        &\leq \dK{-C\frac{2q}{q_e}}\delta^{-(1 - \frac{2}{q_e})}\lpnorm{\Sl f_{\tau_{\ell - 1}}}{\infty}{A_{\tau_{\ell - 1}}^{-}}^{\frac{2}{q_e}(q_e - q)} \Big(\sum_{\substack{\tau_{\ell}, \tau_{\ell}' \subset \tau_{\ell - 1} \\ \text{dist}(\tau_{\ell} , \tau_{\ell}') \gtrsim \la\dK{\ell}}} \iint   |\Sl f_{\tau_{\ell}}(x,t)\Sl f_{\tau_{\ell}'}(x,t)|^{\frac{q}{2}}dxdt \Big)^{\frac{2}{q_e}} \\
        &\lesssim_{\varepsilon} \Big(\big(\la \delta^{-1}\big)^{\frac{n-1}{4}}\dK{\frac{n-1}{2}(\ell - 1)}\Big)^{\frac{2}{q_e}(q_e - q)} \Big(\la^{q\frac{n-1}{2} - n + \varepsilon} \dK{\ell(q\frac{n-1}{2} - n -1)} \delta^{-1}\Big)^{\frac{2}{q_e}} \dK{-C\frac{2q}{q_e}}\delta^{-(1 - \frac{2}{q_e})}\Big(\sum_{\substack{\tau_{\ell},\tau_{\ell}' \subset \tau_{\ell - 1} \\ dist(\tau_{\ell}, \tau_{\ell}') \gtrsim \la \dK{\ell}}} \altlpnorm{f_{\tau_{\ell}}}{2}^{\frac{q}{2}}\lVert f_{\tau_{\ell}'}\rVert_{2}^{\frac{q}{2}}\Big)^{\frac{2}{q_e}} \\
        & \leq \Big(\la^{q\frac{n-1}{2} - n + \varepsilon} \dK{\ell(q\frac{n-1}{2} - n -1)}\Big)^{\frac{2}{q_e}}  \dK{-C\frac{2q}{q_e}} \delta^{-1} \Big(\big(\la \delta^{-1}\big)^{\frac{n-1}{4}}\dK{\frac{n-1}{2}(\ell - 1)}\Big)^{\frac{2}{q_e}(q_e - q)} \big( \#\{\tau_{\ell} \subset \tau_{\ell - 1}\} \big)^{\frac{4}{q_e}} \altlpnorm{f_{\tau_{\ell - 1}}}{2}^{\frac{q}{q_e}} \\
        & \lesssim \Big(\la^{q\frac{n-1}{2} - n + \varepsilon} \dK{\ell(q\frac{n-1}{2} - n -1)}\Big)^{\frac{2}{q_e}}  \dK{-C\frac{2q}{q_e}} \delta^{-1} \Big(\big(\la \delta^{-1}\big)^{\frac{n-1}{4}}\dK{\frac{n-1}{2}(\ell - 1)}\Big)^{\frac{2}{q_e}(q_e - q)} \dK{-\frac{4}{q_e}(n-1)} \altlpnorm{f_{\tau_{\ell - 1}}}{2}^2.
    \end{split}
\end{equation}
Note that 
\begin{equation*}
    \tfrac{2}{q_e}(q_e - q) = \tfrac{12}{n(n-1)} \ \ \implies \ \ \tfrac{n-1}{4} \cdot \tfrac{2}{q_e}\big(q_e  - q \big) = \tfrac{3}{n},
\end{equation*}

\begin{equation*}
    q\tfrac{n-1}{2} - n = \tfrac{n-2}{n} \ \ \implies \ \ \Big(q \tfrac{n-1}{2} - n\Big)\tfrac{2}{q_e} =  \tfrac{n^2 - 5n + 6}{n(n-1)}
\end{equation*}
and 
\begin{equation*}
    q\tfrac{n-1}{2} - n - 1 = -\tfrac{2}{n} \ \ \implies \ \ \big(q \tfrac{n-1}{2} - n - 1\big)\tfrac{2}{q_e} =  -\tfrac{2(n-3)}{n(n-1)}.
\end{equation*}
\noindent 
Collecting the above arithmetic for the exponents  we have
\begin{equation}
    \begin{split}
        II &\lesssim_{\varepsilon} \la^{\frac{n^2 - 5n + 6}{n(n-1)} + \frac{3}{n} + \frac{2}{q_e}\varepsilon}\dK{\ell\big( \frac{6}{n} - \frac{2(n-3)}{n(n-1)}\big)}\delta^{-1 - \frac{3}{n}}\dK{-C'} \lpnorm{f_{\tau_{\ell - 1}}}{2}{\T{n-1}}^2 \\
        &\leq \la^{\frac{n^2 - 2n + 2}{n(n-1)} + \frac{2}{q_e}\varepsilon} \delta^{-1 - \frac{3}{n}}\dK{-C'} \lpnorm{f_{\tau_{\ell - 1}}}{2}{\T{n-1}}^2 \\
        &= \la^{1 - \frac{n - 3}{n(n-1)} + \frac{2}{q_e}\varepsilon} \delta^{-1 - \frac{3}{n}}\dK{-C'} \lpnorm{f_{\tau_{\ell - 1}}}{2}{\T{n-1}}^2
    \end{split}
\end{equation}
since $\frac{6}{n} - \frac{2(n-3)}{n(n-1)} > 0$ and $\dK{\ell} < 1$ for all $1\leq \ell \leq K$.

Thus, we have proved the following proposition.
\begin{proposition}
    Fix $\kappa > 0$. Then for any $\delta \geq \la^{-1 + \kappa}$ and for any $\varepsilon  > 0$ we have
    \begin{equation}
    \begin{split}
        \mixlpnorm{\Sl f_{\tau_{\ell - 1}}}{2}{q_e}{A^{-}_{\tau_{\ell - 1}}} &\leq C\Big(\sum_{\tau_{\ell} \subset \tau_{\ell - 1}}\mixlpnorm{\Sl f_{\tau_{\ell}}}{2}{q_e}{\Tn \times [0, \delta^{-1}]}^2\Big)^{\frac{1}{2}} \\
        &+ C_{\varepsilon}\la^{\frac{1}{2} - \frac{n-3}{2n(n-1)}+\varepsilon}\delta^{-\frac{n+3}{2n}}\dK{-C}\lpnorm{f_{\tau_{\ell - 1}}}{2}{\Tn}.
    \end{split}
    \end{equation}
\end{proposition}

\subsection{Proof of Theorem 1.2}
Note that our large height estimate (\ref{large height endpoint estimate}) and Proposition 8.1 combine to yield the following analog of Proposition 3.1.
\begin{proposition}
    Fix $\kappa > 0$. Then for any $\delta \geq \la^{-1 + \kappa}$ and for any $\varepsilon  > 0$ we have
    \begin{equation}
    \begin{split}
        \mixlpnorm{\Sl f_{\tau_{\ell - 1}}}{2}{q_e}{\Tn \times [0,\delta^{-1}]} &\leq C_{\varepsilon}\big(\la^{\frac{1}{2}} +  \la^{\frac{1}{2} - \frac{n-3}{2n(n-1)}+\varepsilon}\delta^{-\frac{n+3}{2n}}\dK{-C}\big)\lpnorm{f_{\tau_{\ell - 1}}}{2}{\Tn} \\ 
        &+ C'\Big(\sum_{\tau_{\ell} \subset \tau_{\ell - 1}}\mixlpnorm{\Sl f_{\tau_{\ell}}}{2}{q_e}{\Tn \times [0, \delta^{-1}]}^2\Big)^{\frac{1}{2}}.
    \end{split}
    \end{equation}
\end{proposition}
We now turn to the Proof of Theorem 1.2. As with the proof of Theorem 1.1, the proof follows from applying Proposition 8.2 for $1 \leq \ell < K$. We also still need an estimate when $\ell = K$. Let $\tau \subset Q_{\la}$ be a cube of side-length $2\la \delta$. Let $\chi_{\tau}$ be defined as in Section 6. Then we again have
\begin{equation}
    \Sl \chi_{\tau}(D_x)f(x, t) = S_{\la, t} \chi_{\tau}(D_x)f(x) 
\end{equation}
and we can use Corollary 6.0.1 to conclude that
\begin{equation}
    \lpnorm{S_{\la, t}\chi_{\tau}(D_x) \big( S_{\la,s} \chi_{\tau}(D_x)\big)^{*} f}{\infty}{\Tn} \lesssim\la^{\frac{n-1}{2}}|t-s|^{-\frac{n-1}{2}}\lpnorm{f}{1}{\Tn}.
\end{equation}
Also, by Parseval's Theorem
\begin{equation}
    \lpnorm{S_{\la, t}\chi_{\tau}(D_x)f}{2}{\Tn} \lesssim \lpnorm{f}{2}{\Tn}.
\end{equation}
Thus, we can use a simple rescaling argument and the Keel-Tao theorem from \cite{KT} to conclude that
\begin{equation} \label{endpoint K scale estimate}
     \mixlpnorm{\Sl \chi_{\tau}(D_x)f}{2}{q_e}{\Tn \times [0, \delta^{-1}]} \lesssim \la^{\frac{1}{2}}\lpnorm{f}{2}{\Tn}.
\end{equation}
See the proof of Proposition 2.1 in \cite{HSSchro} or the proof of (2.51) in \cite{HuangSoggeStrichartz} for more details.

Now, suppose $\tau_{K-1} \subset Q_{\la}$ is a cube of side-length $2\la \dK{K-1}$. Then by Proposition 8.2 we have
\begin{equation} \label{endpoint broad-narrow at scale K}
\begin{split}
    \mixlpnorm{\Sl f_{\tau_{K-1}}}{2}{q_e}{\Tn\times[0, \delta^{-1}]} \leq C_{\varepsilon}\big(\la^{\frac{1}{2}} +  \la^{\frac{1}{2} - \frac{n-3}{2n(n-1)}+\varepsilon}\delta^{-\frac{n+3}{2n}}\dK{-C}\big)\altlpnorm{f_{\tau_{k-1}}}{2} \\
    + C'\Big(\sum_{\tau_{K} \subset \tau_{K-1}}\mixlpnorm{\Sl f_{\tau_{K}}}{2}{q_e}{\Tn \times [0, \delta^{-1}]}^{2}\Big)^{\frac{1}{2}}
\end{split}
\end{equation}
\noindent
By (\ref{endpoint K scale estimate}) we have
\begin{equation}\label{endpoint smallest scale Strichartz}
    \mixlpnorm{\Sl f_{\tau_{K}}}{2}{q_e}{\Tn \times [0, \delta^{-1}]} = \mixlpnorm{\Sl \chi_{\tau_{K}}(D_x)f_{\tau_{K}}}{2}{q_e}{\Tn \times [0, \delta^{-1}]}   \lesssim \la^{\frac{1}{2}}\altlpnorm{f_{\tau_{K}}}{2}
\end{equation}
Combining this with (\ref{endpoint broad-narrow at scale K}) and using the orthogonality of the functions $\{ f_{\tau_{K}}\}_{\tau_K \subset \tau_{K-1}}$ we have
\begin{equation}
    \mixlpnorm{\Sl f_{\tau_{K-1}}}{2}{q_e}{\Tn\times[0, \delta^{-1}]} \lesssim_{\varepsilon}\big(\la^{\frac{1}{2}} +  \la^{\frac{1}{2} - \frac{n-3}{2n(n-1)}+\varepsilon}\delta^{-\frac{n+3}{2n}}\dK{-C}\big)\altlpnorm{f_{\tau_{K-1}}}{2}.
\end{equation}

Now, suppose we know that for some $1 \leq \ell < K$
\begin{equation} \label{endpoint scale l estimate}
    \mixlpnorm{\Sl f_{\tau_{\ell}}}{2}{q_e}{\Tn\times[0, \delta^{-1}]} \leq C_{\varepsilon}\big(\la^{\frac{1}{2}} +  \la^{\frac{1}{2} - \frac{n-3}{2n(n-1)}+\varepsilon}\delta^{-\frac{n+3}{2n}}\dK{-C}\big)\altlpnorm{f_{\tau_{\ell}}}{2}
\end{equation}
for any cube $\tau_{\ell} \subset Q_{\la}$ of side-length $2\la \dK{\ell}$. Fix a cube $\tau_{\ell - 1} \subset Q_{\la}$. Using Proposition 8.2, (\ref{endpoint scale l estimate}) and the orthogonality of the functions $\{f_{\tau_{\ell}}\}_{\tau_{\ell}\subset \tau_{\ell - 1}}$ we have
\begin{equation}
\begin{split}
    \mixlpnorm{\Sl f_{\tau_{\ell-1}}}{2}{q_e}{\Tn\times[0, \delta^{-1}]} \leq C_{\varepsilon} \big(\la^{\frac{1}{2}} +  \la^{\frac{1}{2} - \frac{n-3}{2n(n-1)}+\varepsilon}\delta^{-\frac{n+3}{2n}}\dK{-C}\big)\altlpnorm{f_{\tau_{\ell-1}}}{2} \\
    + C'\Big(\sum_{\tau_{\ell} \subset \tau_{\ell-1}}\mixlpnorm{\Sl f_{\tau_{\ell}}}{2}{q_e}{\Tn \times [0, \delta^{-1}]}^{2}\Big)^{\frac{1}{2}} \\
    \lesssim_{\varepsilon} \big(\la^{\frac{1}{2}} +  \la^{\frac{1}{2} - \frac{n-3}{2n(n-1)}+\varepsilon}\delta^{-\frac{n+3}{2n}}\dK{-C}\big)\altlpnorm{f_{\tau_{\ell-1}}}{2}
\end{split}
\end{equation}
We once again iterate this argument to yield that for any function $ f(x) = \sum_{k \in \tilde{Q_{\la}}}\widehat{f}(k)e(x \cdot k)$ we have
\begin{equation}
    \mixlpnorm{\Sl f}{2}{q_e}{\Tn \times [0, \delta^{-1}]} \leq C_{\varepsilon, K}\big(\la^{\frac{1}{2}} +  \la^{\frac{1}{2} - \frac{n-3}{2n(n-1)} + \varepsilon}\delta^{-\frac{n+3}{2n}}\dK{-C}\big)\altlpnorm{f}{2}.
\end{equation}
As in Section 7, by choosing $K$ large enough (depending on $\varepsilon$) we have $\dK{-C} < \la^{\varepsilon}$. Further, 
\begin{equation}
    \la^{\frac{1}{2} - \frac{n-3}{2n(n-1)} + \varepsilon}\delta^{-\frac{n+3}{2n}}\dK{-C} \leq \la^{\frac{1}{2} - \frac{n-3}{2n(n-1)} + 2\varepsilon}\delta^{-\frac{n+3}{2n}} \leq \la^{\frac{1}{2}}, \ \text{if} \ \delta \geq \la^{-\frac{n-3}{(n-1)(n+3)} + \frac{4n}{n+3}\varepsilon}.  
\end{equation}

Thus, for any $\varepsilon > 0$, if $\delta \geq \la^{-\frac{n-3}{(n-1)(n+3)} + \varepsilon}$ then 
\begin{equation}
    \mixlpnorm{\Sl f}{2}{q_e}{\Tn \times [0, \delta^{-1}]} \lesssim_{\varepsilon} \la^{\frac{1}{2}}\lpnorm{f}{2}{\Tn}
\end{equation}
and by rescaling this implies
\begin{equation}
    \mixlpnorm{e^{-it \laplac{\Tn}}\beta(D_x/\la)f}{2}{q_e}{\Tn \times [0, \frac{1}{\la \delta}]} \lesssim_{\varepsilon} \lpnorm{f}{2}{\Tn}.
\end{equation}
This completes the proof of Theorem 1.2 on the square torus.
\qed

\section{Generalizing the results to Rectangular Tori}

Consider a rectangular torus $ \Tn_{\alpha_1...\alpha_{n-1}} = \prod_{i = 1}^{n-1} \left(\R{}/2\pi\alpha_i \Z{}\right)$ where $\alpha_i \geq 1$ for $i = 1,..., n-1$. For $x \in \Tn_{\alpha_1...\alpha_{n-1}}$ we can write $x = (\alpha_1y_1,...,\alpha_{n-1}y_{n-1})$ where $y = (y_1,...y_{n-1}) \in \Tn$ and $\Tn = \left( \R{}/2\pi\Z{}\right)^{n-1}$ is once again the standard square torus. Let $B$ be the $n-1 \times n-1$ matrix with the $\alpha_i$ along the diagonal and $0$ everywhere else and let $A = \left(B^{-1}\right)^{t}$. Note then that $A$ is the $n-1 \times n-1$ matrix with $(\alpha_i)^{-1}$ along the diagonal and 0 everywhere else. Note that for any function $g \in L^2(\Tn_{\alpha_1...\alpha_{n-1}})$ we can write $g(x) = f(y)$ where $y \in \Tn$ and $f = g \circ B$ is a function in $L^2(\Tn)$. We can expand $f$ in a Fourier series on $\Tn$
\begin{equation*}
    g(x) = f(y) = \sum_{k\in \Zn}\widehat{f}(k)e(y \cdot k).
\end{equation*}
Further, we have $D_x = A D_y$ and
\begin{equation}
    -\laplac{\Tn_{\alpha_1...\alpha_{n-1}}}g(x) = \sum_{k \in \Zn}\langle A^tAk,k\rangle\widehat{f}(k)e(y \cdot k) = \sum_{k \in \Zn}\widehat{f}(k)Q(k)e(y \cdot k) = Q(D_y)f(y)
\end{equation}
where $Q$ is given by $Q(\xi) = \langle A^tA\xi, \xi\rangle = \sum_{i = 1}^{n-1}\theta_i\xi_i^2$ and $\theta_i = (\alpha_i)^{-2}$. Thus, to prove Theorems 1.1 1.2 on $\Tn_{\alpha_1...,\alpha_{n-1}}$ it suffices to show that for any $f \in L^2(\Tn)$ of the form $f(y) = \sum_{k \in \tilde{Q_{\la}}}\widehat{f}(k)e(y\cdot k)$ we have
\begin{equation} \label{critical rectangular strichartz}
    \lpnorm{e^{itQ(D_y)}\beta(D_y/\la)f}{q_c}{\Tn \times [0,\frac{1}{\la \delta}]]} \lesssim_{\varepsilon} \lpnorm{f}{2}{\Tn} \ \ \ \text{if} \ \ \delta \geq \la^{-\frac{1}{n+1} + \varepsilon}
\end{equation}
and
\begin{equation}\label{endpoint rectangular strichartz}
    \mixlpnorm{e^{itQ(D_y)}\beta(D_y/\la)f}{2}{q_e}{\Tn \times [0,\frac{1}{\la \delta}]]} \lesssim_{\varepsilon} \lpnorm{f}{2}{\Tn} \ \ \ \text{if} \ \ \delta \geq \la^{-\frac{n-3}{(n-1)(n+3)} + \varepsilon}
\end{equation}
for any $\varepsilon > 0$. Let $\eta$ be the bump function defined in Section 3 and define 
\begin{equation} \label{rectangular time-dilated schrodinger propagator}
    \Sl^{Q}f(y,t) = \eta(\delta t)e^{i\la^{-1}tQ(D_y)}\beta(D_y/\la)f(y).
\end{equation}
Then, by rescaling, (\ref{critical rectangular strichartz}) and (\ref{endpoint rectangular strichartz}) are equivalent to showing that 
\begin{equation} 
    \mixlpnorm{\Sl^{Q} f}{p}{q}{\Tn \times [0,\delta^{-1}]} \lesssim \la^{\frac{1}{p}}\lpnorm{f}{2}{\Tn}
\end{equation}
for $(p,q) = (q_c,q_c), (2,q_e)$ and for the same ranges of $\delta$ as in (\ref{critical rectangular strichartz}) and (\ref{endpoint rectangular strichartz}).

Note that the phase function $\xi \mapsto Q(\xi)$ scales the same as $\xi \mapsto |\xi|^2$. Upon examining the proofs of Theorems 1.1 and 1.2 for the square torus, we see that we would be able to replicate our arguments for that special case if we had the following three lemmas. The first lemma is the analog of the discrete bilinear Fourier extension estimate in Lemma 5.5.
\begin{lemma}
    Fix $q = \frac{2(n+2)}{n}$ and let $\omega_1, \omega_2$ be two cubes in $[-1,1]^{n-1}$ with side length $\sim D > 0$ such that dist$(\omega_1, \omega_2) \gtrsim D$. Let $R \gtrsim 1$ with $R^{-1} \lesssim D$ and let $B_{R}$ be any ball of radius $R$ in $\mathbb{R}^n$. Let $\Lambda_i = \{(\xi, Q(\xi)) :   \xi \in \omega_i \}$, $i  =1,2$ be two collections of $R^{-1}$ separated points on the hypersurface associated to $Q$ and let $\Xi_i = \{\xi \in \omega_i : (\xi, Q(\xi)) \in \Lambda_i \}$. Then for any $\varepsilon > 0$ we have
    \begin{equation} \label{bilinear exponential sum estimate for Q}
        \norm{\Big| \prod_{i = 1}^{2} \sum_{\xi \in \Xi_i}a_{\xi} e(x \cdot \xi + tQ(\xi)) \Big|^{\frac{1}{2}}}_{L_{x,t}^{q}(B_R)} \lesssim_{\varepsilon} R^{\frac{n-1}{2} + \varepsilon} D^{\frac{n-1}{2} - \frac{n+1}{q} + \varepsilon} \Big( \prod_{i = 1}^{2} \litlpnorm{a_\xi}{2}{\Xi_i} \Big)^{\frac{1}{2}}.
    \end{equation}
\end{lemma}
\noindent This estimate follows from the arguments in Section 5.1 and the bilinear restriction estimate

\begin{equation}\label{bilinear restriction for Q}
    \lpnorm{\big|E^{Q}_{\omega_1}fE^{Q}_{\omega_2}f \big|^{\frac{1}{2}}}{q}{B_{R}} \lesssim_{\varepsilon} R^{\varepsilon}\left( \lpnorm{f}{2}{\omega_1}\lpnorm{f}{2}{\omega_2}\right)^{\frac{1}{2}}
\end{equation}
where $\omega_1$ and $\omega_2$ are two cubes in $[-1,1]^{n-1}$ with $dist(\omega_1, \omega_2) > 0$ and 
\begin{equation*}
    E^{Q}_{\omega_i}f(x,t) = \int_{\omega_i}f(\xi)e(x\cdot \xi + tQ(\xi))d\xi.
\end{equation*}
Note that (\ref{bilinear restriction for Q}) follows directly from Theorem 5.3 and a straightforward argument involving a change of variables. The implicit constant in the estimate depends on the parameters $\alpha_i$ (i.e. it depends on the particular rectangular torus we are working over) but this is allowed. 

The second lemma is the analog of our one-dimensional kernel estimate from Section 6 and  follows directly from Proposition 6.2. In the following, let $\phi$ and $\chi_{I}$ be the bump functions defined in Section 6.
\begin{lemma}
    Let $\kappa > 0$ and suppose $\delta \geq \la^{-1 + \kappa}$. Then for $x,y \in \T{}$ and $\theta \in (0,1]$ we have
\begin{equation} 
     \big|\eta(\delta t)\eta(\delta s)\big( \chi_{I}^{2}(D_x)\phi^2(D_x / \la)e^{-i\la^{-1}\theta(t-s)\laplac{\T{}}} \big)(x,y)\big| \lesssim \la^{\frac{1}{2}}|t-s|^{-\frac{1}{2}}\big(1 + \dK{m}|t-s|\big).
\end{equation} 
Thus, when $m = K$ (so that $\dK{m} = \delta$), since $|t-s| \leq 2\delta^{-1}$ we have 
\begin{equation}
    \big|\eta(\delta t)\eta(\delta s)\big( \chi_{I}^{2}(D)\phi^2(D / \la)e^{-i\la^{-1}\theta(t-s)\laplac{\T{}}} \big)(x,y) \big| \lesssim \la^{\frac{1}{2}}|t-s|^{-\frac{1}{2}}.
\end{equation}
\end{lemma}
Finally, our third lemma follows from Lemma 9.2 using the same argument we used to prove that Proposition 6.2 implies Corollary 6.0.1 and gives us the $n-1$ dimensional kernel estimates we require. Let $\chi_{\tau}$ and $\beta$ be the bump functions defined in Section 6.
\begin{lemma}
    Let $\kappa > 0$ and suppose $\delta \geq \la^{-1 + \kappa}$. If $x,y \in \T{n-1}$ then
    \begin{equation}
        \big|\eta(\delta t)\eta(\delta s)\big( \chi_{\tau}^{2}(D_x)\beta^2(D_x / \la)e^{i\la^{-1}(t-s)Q(D_x)} \big)(x,y)\big| \lesssim \la^{\frac{n-1}{2}}|t-s|^{-\frac{n-1}{2}}\big(1 + \dK{m}|t-s|\big)^{n-1}.
    \end{equation}
Thus, when $m = K$ (so that $\dK{m} = \delta$), since $|t-s| \leq 2\delta^{-1}$ we have 
\begin{equation}
    \big|\eta(\delta t)\eta(\delta s)\big( \chi_{\tau}^{2}(D_x)\beta^2(D_x / \la)e^{i\la^{-1}(t-s)Q(D_x)} \big)(x,y) \big| \lesssim \la^{\frac{n-1}{2}}|t-s|^{-\frac{n-1}{2}}.
\end{equation}
\end{lemma}
\noindent

Using these three lemmas and following our arguments in Sections 3-7 for the exponent pair $(q_c, q_c)$ and Section 8 for the exponent pair $(2,q_e)$ we attain the following two estimates:
\begin{equation}
    \lpnorm{\Sl^{Q}f}{q_c}{\Tn \times [0,\delta^{-1}]} \lesssim_{\varepsilon} \la^{\frac{1}{q_c}} \lpnorm{f}{2}{\Tn} \ \ \ \text{if} \ \ \delta \geq \la^{-\frac{1}{n+1} + \varepsilon}
\end{equation}
and 
\begin{equation}
    \mixlpnorm{\Sl^{Q}f}{2}{q_e}{\Tn \times [0,\delta^{-1}]} \lesssim_{\varepsilon} \la^{\frac{1}{2}} \lpnorm{f}{2}{\Tn} \ \ \ \text{if} \ \ \delta \geq \la^{-\frac{n-3}{(n-1)(n+3)} + \varepsilon}.
\end{equation}
As stated earlier, these estimates are equivalent to (\ref{critical rectangular strichartz}) and (\ref{endpoint rectangular strichartz}) which in turn yield Theorems 1.1 and 1.2 on $\Tn_{\alpha_1...\alpha_{n-1}}$.

\begin{remark}
    By carefully examining our proofs of Theorems 1.1 and 1.2 we see that for any Keel-Tao admissible exponent pair $(p,q)$ with $p \leq q$ (this occurs in the range where $q_c \leq q \leq q_e$) our methods can prove a lossless Strichartz estimate on a time interval with length depending on $p$ and $q$ which is shorter than the interval given by Theorem 1.1 and longer than the interval given by Theorem 1.2. We chose not to explore this here and focused on the critical exponent $q_c$ and the endpoint exponent pair $(2,q_e)$ for the sake of clarity and due to the historical importance of these specific exponents.
\end{remark}

\begin{remark}
    When $\Tn_{\alpha_1...\alpha_{n-1}}$ is a generic rectangular irrational torus then one could hope to improve the lower bounds for $\delta$ given in Theorems 1.1 and 1.2 by using the stronger dispersive properties of Schr\"{o}dinger waves on these manifolds. This dispersion is quantified by improved kernel estimates for the Sch\"{o}dinger propagators. For example, see Proposition 4.6 in \cite{Deng2017Strichartz}. Better kernel estimates would allow us to perform our height splitting at a smaller height, leading to improvements over Theorems 1.1 and 1.2.
\end{remark}

\begin{remark}
    Another possible avenue to improve the results given in Theorems 1.1 and 1.2 is to show that the results continue to hold on any flat torus. In this case, we can reduce the desired estimates to estimates for the time-dilated Schr\"{o}dinger propagator given in (\ref{rectangular time-dilated schrodinger propagator}) where $Q$ is now a general positive definite quadratic form. There is a bilinear restriction estimate for such phases analogous to (\ref{bilinear restriction for Q}) and one could likely prove the analog of Lemma 9.3 by using similar stationary phase arguments to those used in the proof of the one-dimensional estimates given in Proposition 6.2. However, as we just indicated, the results would not follow directly from the arguments we have presented here and we choose not to pursue them further.
\end{remark}

\subsection{Proof of Corollary 1.2.1} 
Let $\T{d}_{\alpha_1...\alpha_d}$ be a $d$-dimensional rectangular torus and let $q_e = \frac{2d}{d-2}$, as before. The spectral projection operator onto a $\delta$-window around the spectral parameter $\la  \gg 1$ is the multiplier operator defined via the spectral theorem
\begin{equation}
    P_{\la,\delta} = \mathbbm{1}_{A_{\la,\delta}}\left( \sqrt{-\laplac{\Tn_{\alpha_1...\alpha_d}}}\right)
\end{equation}
where $A_{\la, \delta} = [\la - \delta, \la + \delta]$. We now present the short proof that Theorem 1.2 implies Corollary 1.2.1. We remark again that the fact that a lossless Strichartz estimate at the endpoint $\left(2,q_e\right)$ on a time interval of length $\sim \frac{1}{\la\delta}$ implies the estimate (\ref{main spectral projection estimate}) is well-known. For example, see \cite{BlairHuangSoggeStrichartz}. However, we include the short proof for the sake of completeness.

Using an argument similar to the one used in the first part of Section 9, one can show that (\ref{main spectral projection estimate}) is equivalent to the following estimate for the Fourier multiplier $\mathbbm{1}_{A_{\la, \delta}}\left(\sqrt{Q(D_x)}\right)$ on the square torus $\T{d}$:
\begin{equation} \label{spectral projection estimate for Q}
    \lpnorm{\mathbbm{1}_{A_{\la, \delta}}\left(\sqrt{Q(D_x)}\right)f}{q_e}{\T{d}} \lesssim_{\varepsilon} (\la\delta)^{\frac{1}{2}}\lpnorm{f}{2}{\T{d}}, \ \ \ \text{if} \ \delta \geq \la^{-\frac{d-2}{d(d+4)} + \varepsilon}.
\end{equation}
Here $Q(\xi) = \sum_{i = 1}^{n-1}\theta_i\xi_i^2$ and $\theta_i = (\alpha_i)^{-2}$, as in the first part of Section 9. 
We also introduce the approximate spectral projection operator defined as
\begin{equation}
    \rho_{\la, \delta}^{Q}(D_x)f(x) = \sum_{k \in \Z{d}}\widehat{f}(k)\rho_{\la, \delta}^{Q}(k)e(x\cdot k)
\end{equation}
where 
\begin{equation}
    \rho_{\la, \delta}^{Q}(\xi) = \rho\big((\la\delta)^{-1}\big(Q(\xi) - \la^2 - \delta^2\big)\big)
\end{equation}
for $\rho \in \mathcal{S}(\R{})$ with $\rho(0) > 0$ and $\supp{\widehat{\rho}}\subset(-\frac{1}{2}, \frac{1}{2})$.

We want to show that (\ref{spectral projection estimate for Q}) holds. Since $\sqrt{Q(\xi)} \in [\la - \delta, \la + \delta]$ implies that $Q(\xi) \in [\la^2 + \delta^2 - 2\la\delta, \la^2 + \delta^2 +  2\la\delta]$ then standard duality and orthogonality arguments show it suffices to prove the following bound for the approximate spectral projection operator:
\begin{equation}
    \lpnorm{\rho_{\la, \delta}^{Q}(D_x)f}{q_e}{\T{d}} \lesssim_{\varepsilon} (\la\delta)^{\frac{1}{2}}\lpnorm{f}{2}{\T{d}}.
\end{equation}
We may also assume that $f$ is a function such that $f = \beta(D_x / \la)f$. By the Fourier inversion formula we have
\begin{equation} \label{fourier inversion approximate spectral projection}
    \rho_{\la, \delta}^{Q}(D_x)f(x) = \frac{\la \delta}{2 \pi}\int\widehat{\rho}(\la \delta t)e(-t(\la^2 + \delta^2))e^{itQ(D_x)}\beta(D_x/\la)f(x)dt. 
\end{equation}
Fix $\varepsilon > 0$ and let $\delta \geq \la^{-\frac{d-2}{d(d+4)} + \varepsilon}$. Using (\ref{fourier inversion approximate spectral projection}) and applying Minkowski's integral inequality followed by the Cauchy-Schwarz inequality, we have
\begin{equation}
\begin{split}
    \lpnorm{\rho_{\la, \delta}^{Q}(D_x)f}{q_e}{\T{d}} &\lesssim \la \delta \int_{|t| \leq \frac{1}{2\la \delta}} \lpnorm{e^{itQ(D_x)}\beta(D_x/\la)f }{q_e}{\T{d}}dt \\
    &\lesssim (\la \delta)^{\frac{1}{2}}\mixlpnorm{e^{itQ(D_x)}\beta(D_x/\la)f}{2}{q_e}{\T{d}\times[-\frac{1}{2\la\delta},\frac{1}{2\la\delta}]} \lesssim_{\varepsilon} (\la \delta)^{\frac{1}{2}}\lpnorm{f}{2}{\T{d}}.
\end{split}
\end{equation}
The last inequality follows from (\ref{endpoint rectangular strichartz}), which, as we noted earlier, is equivalent to Theorem 1.2.
\qed

\bibliographystyle{acm}
\bibliography{aBib}
\newcommand{\Addresses}{{
  \bigskip
  \footnotesize

  (CQ) \textsc{Department of Mathematics, Johns Hopkins University,
    Baltimore, MD 21218}\par\nopagebreak
  \textit{E-mail address: }\texttt{cquinn24@jhu.edu}

  \medskip
}}

\Addresses

\end{document}